\documentclass[reqno,11pt]{amsart}
\newlength{\myhmargin}  
\usepackage[text={6in,8in},centering,letterpaper,dvips, marginparwidth=0.85in]{geometry}
\usepackage{mathtools,amsfonts,amsmath,amssymb,amsthm,verbatim,setspace,mathabx,latexsym}
\usepackage[toc,title]{appendix}
\usepackage{color}
\usepackage{graphicx}
\usepackage{caption}
\usepackage{subcaption}
\usepackage[none]{hyphenat}
\usepackage{hyperref}
\usepackage{tikz-cd}
\usepackage{todonotes}
\usepackage[normalem]{ulem}

\DeclareMathAlphabet{\mathpzc}{OT1}{pzc}{m}{it}
\DeclareSymbolFont{bbold}{U}{bbold}{m}{n}
\DeclareSymbolFontAlphabet{\mathbbold}{bbold}

\newcommand{\symgrp}{\mathrm{S}_N}
\newcommand{\spinc}{\textrm{Spin}^c}
\newcommand{\spc}{\mathfrak{s}}
\newcommand{\F}{\mathbb{F}}%the ground ring
\newcommand{\cfhat}{\widehat{\mathit{CF}}}%Heegaard Floer chain complex
\newcommand{\dhat}{\widehat{\partial}_{\mathit{HF}}}%Heegaard Floer differential
\newcommand{\hfhat}{\widehat{\mathit{HF}}}%Heegaard Floer homology
\newcommand{\acs}{J_{\mathit{HF}}}%Lipshitz almost complex structure
\newcommand{\del}{\partial} %differential
\newcommand{\CCF}{\widehat{\mathcal{CF}}}%the filtered chain complex 
\newcommand{\G}{\textsc{g}}%genus of Heegaard surface from open book decomposition
\newcommand{\N}{\textsc{n}}%the number of arcs
\newcommand{\D}{\mathcal{D}}%domain in a Heegaard diagram
\newcommand{\DA}{\mathcal{D}(A)}%domain associated to a homology class
\newcommand{\pitwo}{\widehat{\pi}_2}%homology classes of maps
\newcommand{\OSz}{\hat{c}(\xi)}%the Ozsvath-Szabo contact class
\newcommand{\Z}{\mathbb{Z}}
\newcommand{\R}{\mathbb{R}}

\newcommand{\pZ}{\mathpzc{Z}}
\newcommand{\pB}{\mathpzc{B}}
\newcommand{\x}{\mathbf{x}}
\newcommand{\y}{\mathbf{y}}
\newcommand{\vv}{\mathbf{v}}
\newcommand{\w}{\mathbf{w}}
\newcommand{\bstheta}{\boldsymbol\theta}

\newcommand{\vx}{\vec{\mathbf{x}}}
\newcommand{\vy}{\vec{\mathbf{y}}}
\newcommand{\vw}{\vec{\mathbf{w}}}
\newcommand{\vz}{\vec{\mathbf{z}}}
\newcommand{\vty}{\vec{\tilde{\y}}}
\newcommand{\vtheta}{\vec{\boldsymbol\theta}}

\newcommand{\bsa}{\boldsymbol\alpha}
\newcommand{\bsb}{\boldsymbol\beta}
\newcommand{\bsg}{\boldsymbol\gamma}
\newcommand{\Mba}{\Sigma,\bsb,\bsa} 

\newcommand{\Mbpa}{\Sigma,\bsb',\bsa}

\newcommand{\Pg}{S} %the page of the open book decomposition
\newcommand{\arcs}{\mathbf{a}} %pairwise disjoint properly embedded collection of arcs
\newcommand{\barcs}{\mathbf{b}}%isotopic collection of arcs
\newcommand{\base}{\mathbf{z}}%basepoints
\newcommand{\B}{\mathcal{B}} %open book decomposition

\newcommand{\mon}{\phi} %the monodromy
\newcommand{\obd}{(\Pg,\mon)}  %abstract open book
\newcommand{\ordn}{o} %naive algebraic torsion for any arc set and sometimes J
\newcommand{\ord}{\mathbf{o}} %algebraic torsion as the minimum over all open books and arc collections
\newcommand{\bbord}{\mathbbold{o}} %algebraic torsion as defined by a random complete set of arcs

\hypersetup{backref,colorlinks=true,citecolor=blue,linkcolor=blue}
\makeatletter
\def\@xfootnote[#1]{%
  \protected@xdef\@thefnmark{#1}%
  \@footnotemark\@footnotetext}
\makeatother

\title{Filtering the Heegaard Floer contact invariant}
\author[Kutluhan]{\c Ca\u gatay Kutluhan}
\author[Mati\'{c}]{Gordana Mati\'{c}}
\author[Van Horn-Morris]{Jeremy Van Horn-Morris}
\author[Wand]{Andy Wand}
\address{Department of Mathematics, University at Buffalo}
\email{kutluhan@buffalo.edu}
\address{Department of Mathematics, University of Georgia}
\email{gordana@math.uga.edu}
\address{Department of Mathematical Sciences, University of Arkansas}
\email{jvhm@uark.edu}
\address{School of Mathematics and Statistics, University of Glasgow}
\email{andy.wand@glasgow.ac.uk}

\thanks{\c Ca\u gatay Kutluhan was supported in part by NSF grant DMS-1360293 and a Simons Foundation grant No.~519352. \\ 
Gordana Mati\'c was supported in part by Simons Foundation grant No.~246461 and NSF grant DMS-1664567.\\
Jeremy Van Horn-Morris was supported in part by Simons Foundation grant No.~279342 and NSF grant DMS-1612412. \\
Andy Wand was supported in part by ERC grant \textsc{geodycon} and EPSRC EP/P004598/1.}
\newtheorem{theorem}{Theorem}[section]
\newtheorem{prop}[theorem]{Proposition}
\newtheorem{lem}[theorem]{Lemma}
\newtheorem{cor}[theorem]{Corollary}
\newtheorem{conj}[theorem]{Conjecture}
\newtheorem{ques}[theorem]{Question}
\theoremstyle{definition}
\newtheorem{defn}[theorem]{Definition}
\newtheorem*{rmk}{Remark}
\newtheorem{stp}{Step}
\numberwithin{equation}{section}
\begin{document}
\sloppy
\bibliographystyle{amsalpha}

\begin{abstract}
We define an invariant of contact structures in dimension three from Heegaard Floer homology. This invariant takes values in the set $\Z_{\geq0}\cup\{\infty\}$. It is zero for overtwisted contact structures, $\infty$ for Stein fillable contact structures, non-decreasing under Legendrian surgery, and computable from any supporting open book decomposition. As an application, we obstruct Stein fillability on contact 3-manifolds with non-vanishing Ozsv\'ath--Szab\'o contact class. 
\end{abstract}
\maketitle
\vspace{0.25in}
%%%%%%%%%%%%%%%%%%

\section{Introduction} 
\label{sec:intro}
%!TEX root = Filtering.tex

The goal of this article is to define an invariant of closed contact 3-manifolds as a refinement of the contact invariant in Heegaard Floer homology, the \emph{Ozsv\'ath--Szab\'o contact class} \cite{OzsvathSzabo4}, and to study some of its properties. Let $M$ be a closed orientable 3-manifold and $\xi$ be a contact structure on $M$. To define our invariant, we start from an open book decomposition of $M$ supporting $\xi$ and a collection of pairwise disjoint, properly embedded arcs on a page of the open book decomposition. From this data we build a filtered chain complex out of the corresponding Heegaard Floer chain complex, whose filtration captures in an algebraic sense the topological complexity of curves counted by the differential. We then consider how far the Ozsv\'ath--Szab\'o contact class survives in the associated spectral sequence. The result is an invariant of the contact manifold, denoted $\ord(M,\xi)$ and read the \emph{spectral order}, or simply \emph{order}, of $(M,\xi)$, taking values in $\Z_{\geq0}\cup\{\infty\}$.

\begin{theorem}
\label{thm:main}
The contact invariant $\ord$ satisfies the following properties: 
\begin{itemize}\leftskip-0.35in
\item $\ord(M,\xi)=0$ if $(M,\xi)$ is overtwisted.
\item $\ord(M,\xi)=\infty$ if $(M,\xi)$ is Stein fillable.
\item $\ord(M,\xi)$ can be detected on an arbitrary supporting open book decomposition of $(M,\xi)$.
\end{itemize}
\end{theorem}

The second bullet point property above follows from the fact that the contact invariant $\ord$ behaves well under Legendrian surgery, giving a map of partially ordered sets from contact manifolds ordered by Stein cobordisms to the set $\Z_{\geq 0} \cup \{\infty\}$ with the usual ordering:

\begin{theorem}
\label{cor:cobord}
The contact invariant $\ord$ is non-decreasing under Legendrian surgery and in particular gives an obstruction to the existence of Stein cobordisms between contact $3$-manifolds. Specifically, if $(M_-,\xi_-)$ and $(M_+,\xi_+)$ are respectively the concave and convex ends of a Stein cobordism, then $\ord (M_-,\xi_-) \leq \ord(M_+,\xi_+)$.
\end{theorem}

Aside from the properties listed in Theorem \ref{thm:main}, the contact invariant $\ord$ behaves well under connected sums. To be more explicit:

\begin{theorem} 
\label{thm:consum}
Let $(M_1,\xi_1)$ and $(M_2,\xi_2)$ be closed contact 3-manifolds. Then their connected sum satisfies $\ord(M_1\# M_2,\xi_1\# \xi_2) = \min\{\ord(M_1,\xi_1), \ord(M_2,\xi_2)\}$.  
\end{theorem}

\noindent This leads to existence of a family of monoids $\ord^k(S)$ in the mapping class group ${\rm Mod}(S,\partial S)$: $\phi\in{\rm Mod}(S,\partial S)$ belongs to $\ord^k(S)$ if and only if $\ord\geq k$ for the contact 3-manifold specified by the open book decomposition $(S,\phi)$ (Corollary \ref{cor:monoid}).

Theorem~\ref{thm:consum} fits into a broader pattern of similar contact connected sum results. Loosely, any measure of rigidity of $(M_1\# M_2,\xi_1\# \xi_2)$, for example all types of symplectic fillability, having a non-vanishing Ozsv\'ath--Szab\'o contact class, or tightness, is the weaker of that property for $(M_1,\xi_1)$ or $(M_2,\xi_2)$ (\cite{Eliashberg}, \cite{CiEli}, \cite{OzsvathSzabo4}, \cite{Colin}).

\begin{comment}
The statements of Theorems \ref{cor:cobord} and \ref{thm:consum} are supported with actual calculations of bounds on our invariant:
\begin{theorem}
\label{thm:ovalues}
An infinite sequence of distinct positive integers is realized by the spectral order of an infinite family of contact $3$-manifolds with vanishing Ozsv\'ath--Szab\'o contact class.
\end{theorem}

\noindent Our proof of Theorem \ref{thm:ovalues} provides explicit examples of contact 3-manifolds realizing infinitely many distinct values of spectral order. In particular, these examples, together with Theorem \ref{cor:cobord}, answer \cite[Question 1]{LatschevWendl}: \emph{Is there a Heegaard Floer theoretic contact invariant that implies obstructions to Stein cobordisms between pairs of contact 3-manifolds whose Ozsv\'ath--Szab\'o invariants vanish?} Theorem \ref{thm:ovalues} also leads to a proof that for an arbitrary compact oriented surface $S$ with at least four boundary components, the monoids $\ord^k(S)$ are mutually distinct for infinitely many distinct values of~$k$ (Corollary \ref{cor:distinctmonoid}).
\end{comment}

It should be noted that our invariant can be thought of as a Heegaard Floer analog of Latschev and Wendl's algebraic torsion \cite{LatschevWendl} defined for contact manifolds of arbitrary dimension in the context of symplectic field theory (SFT).  An analogous version for contact $3$-manifolds was defined by Hutchings in the context of embedded contact homology (ECH) in \cite[Appendix]{LatschevWendl}. In particular, to a closed oriented 3-manifold $M$, a nondegenerate contact $1$-form $\lambda$ on $M$, and a generic almost complex structure $J$ on $\mathbb{R} \times M$ as needed to define the ECH chain complex, Hutchings associates a number in $\Z_{\geq0}\cup\{\infty\}$. The latter is shown to vanish for overtwisted contact structures for all choices of $\lambda$ and $J$, and can be used to obstruct exact symplectic cobordisms. \begin{comment}However, neither Latschev and Wendl's algebraic torsion nor Hutchings's ECH analog is easily computable.\end{comment} 
Our initial definitions follow the ideas of Hutchings's construction, ported to the setting of Heegaard Floer homology (see \cite{KMVHMW} for more on this).

\subsection*{Future considerations} 
In upcoming work in progress \cite{KMVHMW2}, we introduce a method to detect non-vanishing of spectral order, hence tightness, of a contact structure. We apply this method to a family of contact structures with vanishing Ozsv\'ath--Szab\'o contact class. Furthermore, we compute upper bounds on the spectral order of these contact structures and these upper bounds span the range of all positive integers. Next, we would like to show that there is an increasing sequence of positive integers that provides lower bounds on the spectral order of our family of contact structures. These computations would resolve the following conjecture.
\begin{conj}
\label{conj:seqofvalues}
An infinite sequence of distinct positive integers is realized by the spectral order of an infinite family of contact $3$-manifolds with vanishing Ozsv\'ath--Szab\'o contact class.
\end{conj}
\noindent Such a family of examples would provide substance to Theorems \ref{cor:cobord} and \ref{thm:consum}.

A more conceptual question concerns the potential of a converse to the first bullet point of Theorem \ref{thm:main}:
\begin{ques}
\label{Q1}
Suppose that $(M,\xi)$ has vanishing Ozsv\'{a}th--Szab\'{o} contact class. Does $\ord(M,\xi)=0$ imply that $\xi$ is overtwisted?
\end{ques}
\noindent An affirmative answer to this question would have far-reaching consequences in contact and symplectic geometry. Most important of all, it would imply that in dimension three `algebraically overtwisted' contact structures, i.e. those for which all algebraic invariants defined via counting pseudo-holomorphic curves vanish, are exactly those for which an h-principle holds, so are classified by the algebraic topology of their underlying plane fields. All such algebraic invariants are known to vanish for overtwisted structures, due to existence of pseudo-holomorphic curves with minimal topological complexity in the symplectization. Spectral order on the one hand quantifies topological complexity of chains of pseudo-holomorphic curves, and on the other gives a potential interpretation of \emph{consistency} of an open book decomposition (cf. \cite{Wand}), a combinatorial condition equivalent to tightness of the supported contact structure, in the context of pseudo-holomorphic curves. Combined with its computability, these properties of spectral order increase our chances of answering the above question.

Note also that an affirmative answer to Question \ref{Q1} along with the non-decreasing behavior of spectral order under Legendrian surgery, would provide an alternative and more conceptual proof of the following theorem, which has recently been proved by the last author in \cite{Wand2}:
\begin{theorem}
\label{thm:legtight}
Let $\xi$ be a tight contact structure on $M$, and $K\subset M$ be a null-homologuous Legendrian knot. Then, contact $(-1)$-surgery on $K$ produces a 3-manifold with a tight contact structure.
\end{theorem}
\noindent To be more explicit, if $(M,\xi)$ is a closed contact 3-manifold where $\xi$ is a tight contact structure, and $(M',\xi')$ is obtained from $(M,\xi)$ via Legendrian surgery, then the fact that $\ord(M,\xi)>0$ would imply that $\ord(M',\xi')>0$ by the second bullet point of Theorem \ref{thm:main}, which in turn would imply that $\xi'$ is tight by Theorem \ref{thm:main}.

Another question of interest is related to generalizing our invariant to compact contact 3-manifolds with convex boundary. In this regard, our construction of a filtered chain complex out of the Heegaard Floer chain complex readily generalizes to the case of partial open book decompositions introduced in \cite{HKM4}. This allows us to extend the definition of spectral order (Definition \ref{def:algtor}) to compact contact 3-manifolds with convex boundary. This was independently observed by Juh\'asz and Kang who used it to find an upper bound on the spectral order for a closed contact 3-manifold that contains a Giroux torsion domain \cite{JuhaszKang}. Among other things, we would like to compare $\ord$ to Wendl's \emph{planar torsion} \cite{Wendl}. As is stated in \cite[Theorem 6]{LatschevWendl}, planar torsion provides an upper bound to Latschev and Wendl's algebraic torsion. Moreover, planar torsion detects overtwistedness. One could expect a similar relationship between spectral order and Wendl's planar torsion. These are the content of another work in progress by the authors \cite{KMVHMW3}.

\begin{ques}
\label{Q2}
Suppose that the closed contact $3$-manifold $(M,\xi)$ has planar $k$-torsion. Does it imply $\ord(M,\xi)\leq k$?
\end{ques}

\subsection*{Organization} The organization of this article is as follows: 
\begin{itemize}\leftskip0.275in
\item[{\bf Section \ref{sec:definition}}] We provide the definitions required throughout the article, leading to the definition of our contact invariant $\ord$. These include a preliminary version of the latter, denoted $\ordn$, which \emph{a priori} depends on the choices made to define it. 
\item[{\bf Section \ref{sec:independence}}] This section investigates dependence of $\ordn$ on various choices made in its definition. Among these are a choice of the monodromy of an open book decomposition in its isotopy class and a choice of a pairwise disjoint properly embedded collection of arcs on a page of an open book decomposition.
\item[{\bf Section \ref{sec:proof}}] We exhibit several properties of our contact invariant $\ord$, and in doing so prove Theorems \ref{thm:main}, \ref{cor:cobord}, and \ref{thm:consum}.
\item[{\bf Section \ref{sec:example}}] We present a contact structure with non-vanishing Ozsv\'ath--Szab\'o contact class but with finite $\ord$. This implies, by Theorem \ref{thm:main}, that this contact structure is not Stein fillable, a property of this contact structure that was not previously detected by other methods.
\end{itemize}

%%%%%%%%%%%%%%%%%%
\subsection*{Acknowledgements}
The seeds of this project were sown at the ``\emph{Interactions between contact symplectic topology and gauge theory in dimensions 3 and 4}" workshop at Banff International Research Station (BIRS) in 2011. The first three authors would like to thank BIRS and the organizers of that workshop for creating a wonderful atmosphere for collaboration. We also thank John Baldwin for some helpful conversations, Michael Hutchings for generously sharing his thoughts on the ECH analog of algebraic torsion, and Robert Lipshitz for several very helpful correspondences. A significant portion of this work was completed while the first author was a member and the last three authors were visitors at the Institute for Advanced Study (IAS) in Princeton. We thank IAS faculty, particularly Helmut Hofer, and staff for their hospitality. We are also very grateful to the American Institute of Mathematics (AIM). This project benefited greatly from the AIM SQuaRE program.  
%%%%%%%%%%%%%%%%%%
\section{Definitions}
\label{sec:definition}
%!TEX root = Filtering.tex
\subsection{Background}
\label{sec:def-background}
To set the stage, let $M$ be a closed, connected, and oriented $3$-manifold endowed with a co-oriented contact structure $\xi$. It is understood that the orientation on $M$ is induced by $\xi$. A celebrated theorem of Giroux states that there is a 1--1 correspondence between contact structures up to isotopy and open book decompositions up to positive stabilization \cite{Giroux}. An abstract open book decomposition of $M$ is a pair  $(S,\phi)$ where $S$ is a compact oriented surface of genus $g$ with $\textsc{b}$ boundary components, called the \emph{page}, and $\phi$ is an orientation preserving diffeomorphism of $S$ which restricts to identity in a neighborhood of the boundary, called the \emph{monodromy}. The manifold $M$ is diffeomorphic to $S\times[0,1]/\sim$ where $(p,1)\sim(\phi(p),0)$ for any $p\in S$ and $(p,t)\sim(p,t')$ for any $p\in\partial S$ and $t,t'\in[0,1]$. The open book decomposition is said to support the contact structure $\xi$ if there exists a 1-form $\lambda$ such that $\xi=ker(\lambda)$, $\lambda|_{\partial S}>0$, and $d\lambda|_S>0$.

Now fix an abstract open book decomposition $(S,\phi)$ of $M$ supporting $\xi$ and a pairwise disjoint properly embedded collection of arcs $\arcs=\{a_1,\dots,a_\N\}$ on $S$ that contains a basis, that is, a subcollection of arcs cutting $S$ into a polygon. This arc collection together with the monodromy $\phi$ defines a Heegaard diagram $(\Sigma,\{\beta_1,\dots,\beta_\N\},\{\alpha_1,\dots,\alpha_\N\})$ for $-M$ as in \cite[\textsection 3.1]{HKM}. To be more explicit, let $\barcs=\{b_1,\dots,b_\N\}$ be a collection of arcs on $S$ where $b_i$ is isotopic to $a_i$ via a small isotopy satisfying the following conditions:
\begin{itemize}\leftskip-0.35in
\item The endpoints of $b_i$ are obtained from the endpoints of $a_i$ by pushing along $\partial S$ in the direction of the boundary orientation,
\item $a_i$ intersects $b_i$ transversally at one point, $x_i$, in the interior of $S$,
\item Having fixed an orientation of $a_i$, there is an induced orientation on $b_i$, and the sign of the oriented intersection $a_i\cap b_i$ is positive. (see Figure~\ref{fig:standardarcs}.)
\end{itemize}
\begin{figure}[h]
\centering
\includegraphics[width=2.5in]{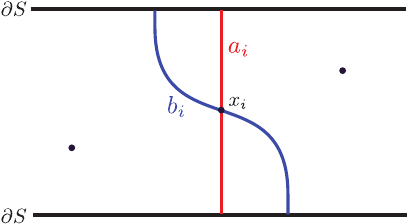}
\caption{The arcs $a_i$ and $b_i$ on the surface $S$.}
\label{fig:standardarcs}
\end{figure}
\noindent Then $\Sigma= S\times\{\frac{1}{2}\}\cup_{\partial S}-S\times\{0\}$, $\alpha_i=a_i\times\{\frac{1}{2}\}\cup a_i\times\{0\}$, and $\beta_i=b_i\times\{\frac{1}{2}\}\cup\phi(b_i)\times\{0\}$. As a parenthetical remark, the Heegaard diagram $(-\Sigma,\{\alpha_1,\dots,\alpha_\N\},\{\beta_1,\dots,\beta_\N\})$ also describes the manifold $-M$, and we may sometimes prefer to use this diagram in figures.

With the preceding understood, we recall the definition of the Heegaard Floer chain complex $(\cfhat(\Mba),\dhat)$. In doing so, we adopt Lipshitz's cylindrical reformulation of Heegaard Floer homology \cite{Lipshitz}. The definition also requires the choice of basepoints $\base\subset\Sigma\smallsetminus\bigcup_{i\in\{1,\dots,\N\}}(\alpha_i\cup\beta_i)$. In the present context, this is done according to the convention in \cite[\textsection 3.1]{HKM}. To be more explicit, place a single basepoint in every connected component of $S\smallsetminus\bigcup_{i\in\{1,\dots,\N\}}a_i$ outside the small strips between $a_i$ and $b_i$ (see Figure \ref{fig:standardarcs}). Following Lipshitz, the chain group $\cfhat(\Mba)$ is freely generated over $\F:=\Z/2\Z$ by \emph{I-chord collections} $\vx:=\x\times[0,1]$ specified by unordered $\N$-tuples of points in $\Sigma$ of the form $\x=\{x_1,\dots,x_\N\}$ where $x_i\in\alpha^{}_i\cap\beta_{\sigma(i)}$ for some element $\sigma$ of the symmetric group $\symgrp$. Given a generic almost complex structure $\acs$ on $\Sigma\times[0,1]\times\R$ satisfying conditions (\textbf{J1})--(\textbf{J5}) in \cite[\textsection 1 p. 959]{Lipshitz}, the differential $\dhat$ on $\cfhat(\Mba)$ is defined to be the endomorphism of $\cfhat(\Mba)$ sending a generator $\vx$ to
\[\sum_{\y}\sum_{A\in\pitwo(\vx,\vy),\;ind(A)=1} n(\vx,\vy;A)\vy.\]
Here $\pitwo(\vx,\vy)$ denotes the set of relative homology classes of continuous maps from a Riemann surface with boundary and boundary punctures into $\Sigma\times[0,1]\times\R$ such that it maps the boundary of the surface into $\boldsymbol{\alpha}\times\{0\}\times\R\cup\boldsymbol{\beta}\times\{1\}\times\R$, it converges to $\vx$ and $\vy$ at its punctures, and it has trivial homological intersection with $\{z\}\times[0,1]\times\R$. Meanwhile, \emph{ind(A)} denotes the index of a class $A\in\pitwo(\vx,\vy)$ (see \cite[Definition 4.4]{Lipshitz}), and $n(\vx,\vy;A)$ is a signed count, modulo $\R$-translation, of $\acs$-holomorphic curves in $\Sigma\times[0,1]\times\R$ satisfying conditions (\textbf{M0})--(\textbf{M6}) in \cite[\textsection 1 p. 960]{Lipshitz} and representing the class $A$. The latter is guaranteed to be finite if we choose the monodromy $\phi$ appropriately in its isotopy class so as to make the multipointed Heegaard diagram $(\Sigma,\bsb,\bsa,\base)$ \emph{admissible}. A multipointed Heegaard diagram is admissible if every non-trivial periodic domain has both positive and negative coefficients (see \cite[Definition 5.1]{Lipshitz}).

\begin{rmk}
Even though Lipshitz carried out his construction of a cylindrical reformulation of Heegaard Floer homology in the case $\N=2g+\textsc{b}-1$ (in other words, the case with one basepoint), the details of his construction and especially the results in \cite[\textsection 4 and \textsection 10]{Lipshitz} carry over to the multipointed case but for cosmetic changes.
\end{rmk}
\subsection{The filtered chain complex}
\label{sec:def-complex}
Next we build a filtered chain complex out of $(\cfhat(\Mba),\dhat)$. To do this, we adopt Hutchings's recipe in \cite[\textsection 6]{Hutchings3}. Given a pair of generators $\vx$ and $\vy$, define a function $J_+$ on $\pitwo(\vx,\vy)$ by
\begin{equation}
\label{eq:jplus-prod}
J_+(A):=\mu(\DA)-2e(\DA)+|\x|-|\y|\footnote[$\dagger$]{The interested reader may refer to \cite{KMVHMW} to see how the authors originally came up with this formula.},
\end{equation}
where $\DA$ is the \emph{domain} in the pointed Heegaard diagram $(\Sigma,\boldsymbol{\beta},\boldsymbol{\alpha},z)$ representing a class $A\in\pitwo(\vx,\vy)$, $\mu(\DA)$ is the \emph{Maslov index} of $\DA$ as in the traditional setting of \cite{OzsvathSzabo}, $e(\DA)$ is the \emph{Euler measure} of $\DA$ (see \cite[\textsection 4.1 p. 973]{Lipshitz} for definition), and $|\cdot|$ denotes the number of cycles in the element of the symmetric group $\symgrp$ associated to a given generator following the convention described above in Section \ref{sec:def-background}. Since the Maslov index and Euler measure are additive under concatenation of domains, so is $J_+$. More precisely, for any $A\in\pitwo(\vx,\vy)$ and $A'\in\pitwo(\vy,\vz)$, we have 
\[J_+(A+A')=J_+(A)+J_+(A').\]
Now suppose that $A\in\pitwo(\vx,\vy)$ is represented by a $\acs$-holomorphic curve $C_L$ in $\Sigma\times[0,1]\times\R$ satisfying conditions (\textbf{M0})--(\textbf{M6}) in \cite[\textsection 1]{Lipshitz}. Then by \cite[Proposition 4.2]{Lipshitz} (cf. \cite[Proposition $4.2'$]{Lipshitz2}),
\begin{equation}
\label{eq:euler-lip}
\chi(C_L)=\N-n_{\x}(\DA)-n_{\y}(\DA)+e(\DA).
\end{equation} 
Here, $n_p(\DA)$ denotes the \emph{point measure}, namely, the average of the coefficients of $\DA$ for the four regions with corners at $p\in\alpha_i\cap\beta_j$. Meanwhile, Lipshitz's formula for the Maslov index of domains \cite[Corollary 4.10]{Lipshitz} (cf.  \cite[Proposition $4.8'$]{Lipshitz2}) asserts that 
\begin{equation}
\label{eq:maslov}
\mu(\DA)=n_{\x}(\DA)+n_{\y}(\DA)+e(\DA).
\end{equation}
Combining, \eqref{eq:euler-lip} and \eqref{eq:maslov}, we obtain
\[\mu(\DA)-2e(\DA)=-\chi(C_L)+\N,\]
and hence \eqref{eq:jplus-prod} can be rewritten as 
\begin{equation}
\label{eq:j-plus-lip}
J_+(A)=-\chi(C_L)+\N+|\x|-|\y|.
\end{equation} 
With the preceding understood, consider the smooth compact oriented surface $C$ obtained from the compactification of $C_L$ by attaching 2-dimensional 1-handles along pairs of points in $\alpha_i\times\{0\}\times\R\cap C_L$ and $\beta_i\times\{1\}\times\R\cap C_L$ for each $i=1,\dots,\N$, and then smoothing. Then $\chi(C)=\chi(C_L)-\N$, and $|\x|$ (resp. $|\y|$) is equal to the number of boundary components of $C$ arising from the $I$-chord $\vx$ (resp. $\vy$). Hence, we can further rewrite \eqref{eq:j-plus-lip} as
\begin{equation}\label{eq:jplus}
J_+(A)=\sum_{C_j\subset C} (2g_j-2 +2|\x_{j}|),
\end{equation}
where each $C_j$ denotes a connected component of $C$, $g_j$ denotes the genus of $C_j$, and each $\x_{j}\subset\x$ denotes the maximal subcollection of points in $\x$ such that $\x_j\times[0,1]$ lies on the boundary of the component $C_j$. Note that each connected component of $C$ has non-empty intersections with the $I$-chord collections specified by $\x$ and $\y$ since each connected component of $C_L$ has non-empty negative and positive ends. Therefore, it follows from \eqref{eq:jplus} that $2\,|\,J_+(A)$, and $J_+(A)\geq 0$.

\begin{rmk}
If there exists an embedded $\acs$-holomorphic curve $C_L$ representing the class $A$, then the Maslov index of $\DA$ agrees with the Fredholm index of $C_L$. For Maslov index-1 domains, we prefer to use the following equivalent formula:
\begin{equation}\label{eq:jplus-2}
J_+(A)=2[n_{\x}(\DA)+n_{\y}(\DA)]-1+|\x|-|\y|.
\end{equation}
\end{rmk}

\subsection{An analog of algebraic torsion}
\label{sec:def-AT}
Following \cite[Appendix]{LatschevWendl}, we decompose the Heegaard Floer differential as
\[\dhat=\partial_0+\partial_1+\cdots+\partial_\ell+\cdots,\]
where $\partial_\ell$ counts $\acs$-holomorphic curves with $J_+=2\ell$ and having empty intersection with $\{z\}\times[0,1]\times\R$. Since $J_+$ is additive under gluing of $J$-holomorphic curves, the above decomposition induces a spectral sequence with pages 
\[E^k(S,\phi,\arcs;\acs)=H_\ast(E^{k-1}(S,\phi,\arcs;\acs),d_{k-1}).\] 
To be more explicit, consider the $\Z$-graded direct sum 
\[\CCF(S,\phi,\arcs):=\cfhat(\Mba)\otimes_\F \F[t,t^{-1}]\] endowed with the endomorphism $\widehat{\partial}$ defined by 
\[\widehat{\partial}(\sum_{i\in\Z}c_it^i):=\sum_{i\in\Z}\big(\sum_{\ell\in\Z}(\partial_\ell c_i) t^{i-\ell}\big).\] 
Here $c_i\neq0$ for only finitely many $i\in\Z$. Note that the additivity property of $J_+$ implies that 
\[\sum_{i+j=\ell}\partial_i\circ\partial_j=0,\]
for any $\ell\geq0$; hence, $\widehat{\partial}\circ\widehat{\partial}=0$ making $(\CCF(S,\phi,\arcs),\widehat{\partial})$ into a filtered chain complex where the $p^{th}$ filtration level 
\[\mathcal{F}^p(S,\phi,\arcs)=\{\sum_{i\leq p}c_it^{i}\;|\;c_i\in\cfhat(\Mba)\}.\] 
Then $(E^k(S,\phi,\arcs;\acs),d_k)$ is the spectral sequence associated to this filtered chain complex where $d_k$ is the restriction of $\widehat{\partial}$ to $E^k(S,\phi,\arcs;\acs)$. To be more explicit, let $A_p^r$ denote the subcomplex defined by 
\[A_p^r=\{c\in \mathcal{F}^p(S,\phi,\arcs)\,|\,\widehat{\partial}c\in \mathcal{F}^{p-r}(S,\phi,\arcs)\}.\]
Then,
\[E_p^k(S,\phi,\arcs;\acs)=\frac{A_p^k}{\widehat{\partial}A_{p+k-1}^{k-1}+A_{p-1}^{k-1}}.\]
A simple computation shows that $E_{p}^k(S,\phi,\arcs;\acs)$ is isomorphic to
\[\frac{\pZ^k(S,\phi,\arcs;\acs)}{\pB^k(S,\phi,\arcs;\acs)},\]
where
\begin{eqnarray*}
\pZ^k(S,\phi,\arcs;\acs)&:=&\{c\in\cfhat(\Mba)\,|\,\exists c_i\in\cfhat(\Mba)\;{\rm for}\,i=1,\dots,k-1\\
&&\hspace{1.2in}{\rm s.t.}\;\del_0c=0\,{\rm and}\,\del_jc=\sum_{i=0}^{j-1}\del_ic_{j-i}\;{\rm for}\;0<j<k\},\\
\pB^k(S,\phi,\arcs;\acs)&:=&\{\sum_{i=0}^{k-1}\del_ib_{i}\,|\,b_i\in\cfhat(\Mba)\;{\rm and}\,\sum_{i-\ell=j}\del_\ell b_i=0\;{\rm for}\;0<j<k\},
\end{eqnarray*}
and since $\mathcal{F}^p(S,\phi,\arcs)\cong\mathcal{F}^{p-1}(S,\phi,\arcs)$, the above quotient completely determines the pages of the spectral sequence.

By \cite[Theorem 3.1]{HKM}, the Heegaard Floer generator specified by the set of distinguished points $\x_\xi=\{x_1,\dots,x_\N\}$ indicated in Figure \ref{fig:standardarcs} represents the Ozsv\'ath--Szab\'o contact class $\OSz\in\hfhat(-M)$, and it satisfies $\del_i\vx_\xi=0$ for all $i\geq0$. This is because there is no Fredholm index-1 $\acs$-holomorphic curve in $\Sigma\times[0,1]\times\R$ satisfying conditions (\textbf{M0})--(\textbf{M6}) in \cite[\textsection 1]{Lipshitz} with $\vx_\xi$ at its negative punctures and having empty intersection with $\{z\}\times[0,1]\times\R$. Hence, $\vx_\xi$ represents a cycle in $E^k(S,\phi,\arcs;\acs)$ for all $k\geq1$, and it is natural to ask how far into the spectral sequence does $\vx_\xi$ survive.
\begin{defn}
\label{def:algktor}
Define $\ordn(S,\phi,\arcs;\acs)$ to be the smallest non-negative integer $k$ such that the generator $\vx_\xi$ represents the trivial class in $E^{k+1}(S,\phi,\arcs;\acs)$. 
\end{defn}
Ideally, one would like to show that $\ordn(S,\phi,\arcs;\acs)$ does not depend on choices of $(S,\phi,\arcs)$ and $\acs$. This is not true in general. For example, consider the closed contact 3-manifold where the contact structure is supported by the open book decomposition $(S,\phi)$ where $S$ is a 4-holed sphere and $\phi$ is the product of Dehn twists depicted in Figure \ref{fig:failure_A}. Using the basis of arcs $\arcs$ shown in Figure \ref{fig:failure_A}, and a generic \emph{split} almost complex structure $\acs$, we observe that the shaded domain $\D$ in Figure \ref{fig:failure_B} is sufficient for the vanishing of the Ozsv\'ath--Szab\'o contact class. A simple computation shows that $J_+(\D)=~2$. Therefore, $\vx_\xi$ represents the trivial class in $E^2(S,\phi,\arcs;\acs)$, and $\ordn(S,\phi,\arcs;\acs)\leq1$. Furthermore, using the symmetry of the open book decomposition and the choice of the arc basis, one can argue as in \cite{KMVHMW2} that $\ordn(S,\phi,\arcs;\acs)=1$. However, the contact structure supported by the open book decomposition $(S,\phi)$ is overtwisted, which can be seen after a sequence of positive stabilizations to reveal the overtwisted disk (see \cite{Wand} for an explicit algorithm). Then there exists another open book decomposition $(S',\phi')$ and a basis of arcs $\arcs'$ on $S'$ for which $\ordn(S',\phi',\arcs';\acs')=0$ using a generic split almost complex structure $\acs'$ (see proof of Theorem \ref{thm:vanishing}). As a result, $\ordn$ is not independent of these choices. 

\begin{figure}[h]
\centering
\begin{subfigure}[b]{0.4\linewidth}
\centering
\includegraphics[width=2in]{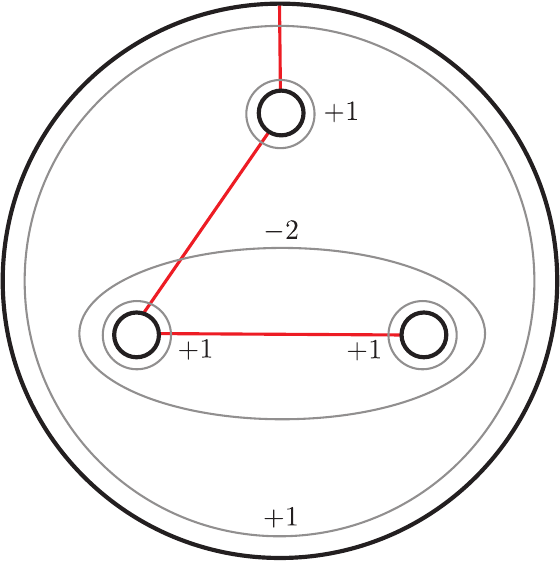}
\caption{}
\label{fig:failure_A}
\end{subfigure}%
\begin{subfigure}[b]{0.4\linewidth}
\centering
\includegraphics[width=2in]{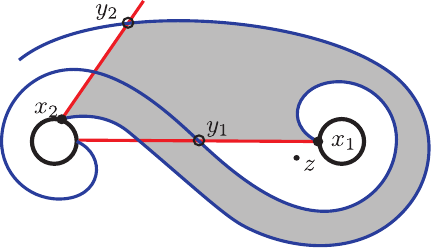}
\caption{}
\label{fig:failure_B}
\end{subfigure}
\centering
\caption{On the left is an open book decomposition $(S,\phi)$ supporting an overtwisted contact structure and a basis of arcs depicted in red. On the right is a Maslov index-1 holomorphic domain with $J_+=2$ in the $S\times\{0\}$ half of the Heegaard diagram $(-\Sigma,\bsa,\bsb)$.}
\label{fig:failure}
\end{figure}

\begin{defn} 
\label{def:algtor}
Let $(M,\xi)$ be a closed contact $3$-manifold. Then define the \emph{spectral order}
\[\ord(M,\xi):=\min\{\ordn(S,\phi,\arcs;\acs)\},\]
where the minimum is taken over all data $(S,\phi,\arcs;\acs)$ such that $(S,\phi)$ is an open book decomposition of $M$ supporting $\xi$, $\arcs$ is pairwise disjoint properly embedded collection of arcs on $S$ that contains a basis, and $\acs$ is a generic almost complex structure on $\Sigma\times[0,1]\times\R$ satisfying conditions (\textbf{J1})--(\textbf{J5}) in \cite[\textsection 1]{Lipshitz}. 
\end{defn}
It follows immediately that Definition \ref{def:algtor} yields an invariant of contact structures. With the definition of our contact invariant in place, the first bullet point of Theorem \ref{thm:main} follows without much effort.
\begin{theorem}
\label{thm:vanishing}
Let $\xi_{\textsc{ot}}$ be an overtwisted contact structure on a closed $3$-manifold $M$. Then $\ord(M,\xi_{\textsc{ot}})=0$. 
\end{theorem}
\begin{proof}
Note that an overtwisted contact structure is supported by an open book decomposition $(S,\phi)$ where the monodromy $\phi$ is not right-veering \cite[Theorem 1.1]{HKM2}. One can find a basis of arcs $\arcs$ on $S$ so that in the corresponding Heegaard diagram $\widehat{\partial}_{\textsc{hf}}\vy=\vx_{\xi_{\textsc{ot}}}$ where $\y=\{y_1,x_2,\dots,x_\G\}$ and there is exactly one Maslov index-1 holomorphic domain $\D$, a bigon, that contributes to the differential \cite[Lemma 3.2]{HKM} as defined by a split complex structure on $\Sigma\times[0,1]\times\R$. Therefore, $n_{\y}(\DA)=\frac{1}{4}$, $n_{\x_{\xi_{\textsc{ot}}}}(\D)=\frac{1}{4}$, $|\y|=\G$, and $|\x_{\xi_{\textsc{ot}}}|=\G$. Applying \eqref{eq:jplus-2}, we find $J_+(\D)=0$. As a result, $\ord(M,\xi_{\textsc{ot}})=0$.
\end{proof}

%%%%%%%%%%%%%%%%%%
\section{Dependence on choices}
\label{sec:independence}
%!TEX root = Filtering.tex

This section investigates dependence of $\ordn(S,\phi,\arcs;\acs)$ on a choice of generic almost complex structure $\acs$ on $\Sigma\times[0,1]\times\R$, where $\Sigma= S\times\{\frac{1}{2}\}\cup_{\partial S}-S\times\{0\}$, a choice of the monodromy $\phi$ in its isotopy class, and how it changes under certain modifications of arc collections. We start with \emph{a priori} dependence of $\ordn$ on a choice of generic almost complex structure.
\subsection{Independence of almost complex structures}
\label{sec:acsind}
\begin{prop}
\label{prop:j-independence}
Fix an open book decomposition $(S,\phi)$ of $M$ supporting $\xi$ and a pairwise disjoint properly embedded collection of arcs $\arcs$ on $S$ that contains a basis. Suppose that $(S,\phi,\arcs)$ yields an admissible Heegaard diagram, and let $J_{\mathit{HF}}^0$ and $J_{\mathit{HF}}^1$ be two generic almost complex structures on $\Sigma\times[0,1]\times\R$ satisfying conditions \textnormal{(\textbf{J1})--(\textbf{J5})} in \cite[\textsection 1]{Lipshitz}. Then $\ordn(S,\phi,\arcs;J_{\mathit{HF}}^0)=\ordn(S,\phi,\arcs;J_{\mathit{HF}}^1)$.
\end{prop}
\begin{proof}
There exists a smooth 1-parameter family of $\R$-invariant almost complex structures $\{J^s_{\mathit{HF}}\}_{s\in\R}$ on $\Sigma\times[0,1]\times\R$ that agrees with $J_{\mathit{HF}}^0$ if $s<\epsilon$ and with $J_{\mathit{HF}}^1$ if $s> 1-\epsilon$ for some $\epsilon\ll 1$. As is explained in \cite[\textsection 9]{Lipshitz}, this family of almost complex structures can be chosen to satisfy conditions (\textbf{J1}),(\textbf{J2}), and (\textbf{J4}) in \cite[\textsection 1]{Lipshitz} when considered as a non-$\R$-invariant almost complex structure on $\Sigma\times[0,1]\times\R$. Furthermore, this almost complex structure guarantees transversality for pseudo-holomorphic curves with prescribed boundary conditions. It is used in \cite[\textsection 9]{Lipshitz} to define a chain map 
\[\Phi:(\cfhat(\Mba),\widehat{\partial}^0_{\mathit{HF}})\to(\cfhat(\Mba),\widehat{\partial}^1_{\mathit{HF}})\] 
via a signed count of $J^s_{\mathit{HF}}$-holomorphic curves in $\Sigma\times[0,1]\times\R$ satisfying conditions (\textbf{M0})--(\textbf{M6}) in \cite[\textsection 1]{Lipshitz} and representing relative homology classes $A\in \pitwo(\vx,\vy)$ with $\mathit{ind}(A)=0$. If $J^s_{\mathit{HF}}$ is generic, then the moduli space of such $J^s_{\mathit{HF}}$-holomorphic curves representing a class $A\in \pitwo(\vx,\vy)$ with $\mathit{ind}(A)=0$, respectively $\mathit{ind}(A)=1$, is a smooth orientable $0$-dimensional, respectively $1$-dimensional, manifold whose compactification in the $1$-dimensional case is obtained by adding on pseudo-holomorphic buildings of height 2 in which one level is $J^s_{\mathit{HF}}$-holomorphic and the other is either $J_{\mathit{HF}}^0$-holomorphic or $J_{\mathit{HF}}^1$-holomorphic as the case may be. The topology of the curves in each component of these moduli spaces is fixed.

Now we define an integer valued function on moduli spaces of $J^s_{\mathit{HF}}$-holomorphic curves in $\Sigma\times[0,1]\times\R$ with $\mathit{ind}\leq 1$ satisfying conditions (\textbf{M0})--(\textbf{M6}) in \cite[\textsection 1]{Lipshitz}. If $C_L$ is such a curve representing a class in $\pitwo(\vx,\vy)$, then define 
\begin{equation}
\label{eq:j-plus-cont}
J_+(C_L):=-\chi(C_L)+\N+|\mathbf{x}|-|\mathbf{y}|.
\end{equation}
Note that \eqref{eq:j-plus-cont} is additive in the sense that if a pseudo-holomorphic building of height 2 consists of a $J_{\mathit{HF}}^0$-holomorphic curve $C^1_L$ with $\mathit{ind}=1$ representing a class in $\pitwo(\vx,\vx')$ and a $J^s_{\mathit{HF}}$-holomorphic curve $C^0_L$ with $\mathit{ind}=0$ representing a class in $\pitwo(\vx',\vy)$, then the $J^s_{\mathit{HF}}$-holomorphic curve $C_L$ obtained from these by gluing (see \cite[Appendix A]{Lipshitz}) satisfies
\begin{equation}
\label{eq:j-plus-add}
J_+(C_L)=J_+(C^1_L)+J_+(C^0_L),
\end{equation}
since $\chi(C_L)=\chi(C^1_L)+\chi(C^0_L)-\N$. The same holds for a pseudo-holomorphic building of height 2 consisting of a $J^s_{\mathit{HF}}$-holomorphic curve $C^0_L$ with $\mathit{ind}=0$ representing a class in $\pitwo(\vx,\vy')$ and a $J_{\mathit{HF}}^1$-holomorphic curve $C^1_L$ with $\mathit{ind}=1$ representing a class in $\pitwo(\vy',\vy)$.
Note also that \eqref{eq:j-plus-cont} coincides with \eqref{eq:j-plus-lip}, which allows us to deduce similarly that $J_+(C_L)$ is a non-negative even integer. Hence, we may decompose $\Phi$ as
\[\Phi=\Phi^0+\Phi^1+\cdots+\Phi^\ell+\cdots,\]
where $\Phi^\ell$ counts $J^s_{\mathit{HF}}$-holomorphic curves with $J_+=2\ell$. Since $\Phi$ is a chain map, and $J_+$ is additive under gluing, it follows that 
\[\sum_{i+j=\ell}(\Phi^i\circ\partial^0_{j}-\partial^1_{i}\circ\Phi^j)=0.\]
This identity implies that there is a filtered chain map $\widehat{\Phi}$ from $(\CCF(S,\phi,\arcs),\widehat{\partial}^0)$ to $(\CCF(S,\phi,\arcs),\widehat{\partial}^1)$ defined by
\[\widehat{\Phi}(\sum_{i\in\Z}c_it^i):=\sum_{i\in\Z}\big(\sum_{\ell\in\Z}(\Phi^\ell c_i) t^{i-\ell}\big),\] 
hence a morphism of spectral sequences from $E^\ast(S,\phi,\arcs;J_{\mathit{HF}}^0)$ to $E^\ast(S,\phi,\arcs;J_{\mathit{HF}}^1)$. Moreover, $\Phi(\vx_\xi)=\vx_\xi$ since the only $J^s_{\mathit{HF}}$-holomorphic curve with negative ends at $\vx_\xi$ satisfying conditions (\textbf{M0})--(\textbf{M6}) in \cite[\textsection 1]{Lipshitz} is $\vx_\xi\times\R$. Therefore, we have $\ordn(S,\phi,\arcs;J_{\mathit{HF}}^0)\geq \ordn(S,\phi,\arcs;J_{\mathit{HF}}^1)$. On the other hand, we may also consider the chain map induced by the smooth $1$-parameter family of almost complex structures $\{J_{\mathit{HF}}^{1-s}\}_{s\in\R}$. Likewise, we obtain $\ordn(S,\phi,\arcs;J_{\mathit{HF}}^0)\leq \ordn(S,\phi,\arcs;J_{\mathit{HF}}^1)$. As a result, $\ordn(S,\phi,\arcs;J_{\mathit{HF}}^0)=\ordn(S,\phi,\arcs;J_{\mathit{HF}}^1)$.
\end{proof}
\subsection{Isotopy independence}
\label{sec:isotopyind}
What with Proposition \ref{prop:j-independence}, we may drop a choice of generic almost complex structure from the notation and simply write $\ordn(S,\phi,\arcs)$. We proceed to discuss the dependence of $\ordn$ on the monodromy. In this regard, let $\phi$ and $\phi'$ be two orientation-preserving diffeomorphisms of $S$ that restrict to the identity in a neighborhood of $\partial S$. Suppose that $\phi$ is isotopic to $\phi'$, and fix an isotopy $\{\phi_t\}_{t\in[0,1]}$ relative to $\partial S$ such that $\phi_0=\phi$ and $\phi_1=\phi'$. Given a pairwise disjoint properly embedded collection of arcs $\arcs$ on $S$ that contains a basis, the isotopy $\{\phi_t\}_{t\in[0,1]}$ yields an isotopy of arcs $\{\phi_t(\barcs)\}_{t\in[0,1]}$, where $\barcs$ is the collection of arcs as in Section \ref{sec:def-background}. Of interest to us are two kinds of isotopies:
\begin{enumerate}\leftskip-0.25in
\item For any $t\in[0,1]$, $\arcs$ intersects $\phi_t(\barcs)$ transversally in the interior of $S$.
\item The isotopy creates/annihilates a pair of transverse intersections between $\arcs$ and $\phi(\barcs)$.
\end{enumerate}
Following \cite{Lipshitz}, we refer to such isotopies as \emph{basic isotopies}. In general, a pointed isotopy between two multipointed Heegaard diagrams, namely, an isotopy supported in the complement of the basepoints, is called \emph{admissible} if each intermediate multipointed Heegaard diagram is admissible. Any two admissible multipointed Heegaard diagrams that are pointed isotopic are in fact isotopic through a sequence of admissible basic isotopies (see \cite[Proposition 5.6]{Lipshitz}). Note that isotopies of the monodromy of an open book decomposition yield pointed isotopies of the corresponding multipointed Heegaard diagram. Therefore, it suffices to investigate the behavior of $\ordn$ under admissible basic isotopies of the monodromy.

\begin{prop}
\label{prop:isotopy}
Let $(S,\phi)$ be an open book decomposition and $\arcs$ be a pairwise disjoint properly embedded collection of arcs $\arcs$ on $S$ that contains a basis. Suppose that $(S,\phi,\arcs)$ yields an admissible multipointed Heegaard diagram, and that $\phi'$ is isotopic to $\phi$ via an admissible basic isotopy. Then $\ordn(S,\phi',\arcs)=\ordn(S,\phi,\arcs)$. 
\end{prop}
\begin{proof}
As is explained in \cite[Chapter 9]{Lipshitz} (cf. \cite[\textsection 7.3]{OzsvathSzabo}), basic isotopies of the first kind above are equivalent to deformations of the complex structure on $\Sigma$. With this understood, $\ordn$ is unchanged under isotopies of this sort by Proposition \ref{prop:j-independence}. As for basic isotopies of the second kind above, we consider the chain maps induced by the multipointed Heegaard triple diagram $(\Sigma,\bsb',\bsb,\bsa,\base)$ where $\bsb'=\{\beta_1',\dots,\beta_{\N}'\}$ is such that each $\beta_i'$ is obtained from a small Hamiltonian isotopy of $a_i\cup\phi'(b_i)$ so that it intersects $\beta_i$ transversally in exactly two points near the point $x_i$ as shown in Figure \ref{fig:isotopy}, while it is disjoint from $\beta_j$ for $j\neq i$. As a result, the Heegaard diagram $(\Sigma,\bsb',\bsb)$ represents the manifold $\#_{\G}S^1\times S^2$; we may assume that $\bsb'$ is sufficiently close to $\bsb$ so that the signed area of the region between them is zero with respect to an area form on $\Sigma$ which delivers the admissibility criteria for the multipointed Heegaard diagram $(\Sigma,\bsb',\bsb,\base)$ as stated in \cite[Lemma 5.3]{Lipshitz}. Consequently, the multipointed Heegaard triple diagram $(\Sigma,\bsb',\bsb,\bsa,\base)$ is also admissible by \cite[Lemma 10.14]{Lipshitz}.

\begin{figure}[h]
\centering
\includegraphics[width=2.75in]{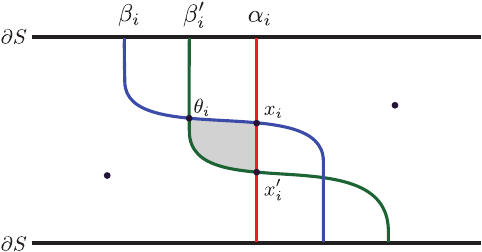}
\caption{Part of the restriction of the multipointed Heegaard triple diagram $(\Sigma,\bsb',\bsb,\bsa,\base)$ to $S\times\{\frac{1}{2}\}\subset\Sigma$.}
\label{fig:isotopy}
\end{figure}

The Heegaard triple diagram $(\Sigma,\bsb',\bsb,\bsa)$ describes a cobordism with one outgoing boundary component and two incoming boundary components, one of which is diffeomorphic to the manifold $\#_{\G}S^1\times S^2$. To be more specific, this cobordism is diffeomorphic to the complement of a tubular neighborhood of a bouquet of $\G$ embedded circles in the product cobordism $[0,1]\times M$. It follows that, there is a unique $\spinc$ structure $\mathfrak{t}_\xi$ on this cobordism which restricts to the trivial $\spinc$ structure $\spc_\circ$ on $\#_{\G}S^1\times S^2$ and to $\spc_\xi$ on $M$. 

With the preceding understood, there exists a chain map 
\[\hat{f}_{\bsb',\bsb,\bsa;\mathfrak{t}_\xi}:\cfhat(\Sigma,\bsb',\bsb,\spc_\circ)\otimes_{\F} \cfhat(\Mba,\spc_\xi)\to \cfhat(\Mbpa,\spc_\xi),\]
defined by counting embedded Fredholm index-$0$ pseudo-holomorphic curves in $\Sigma\times T$ subject to appropriate boundary conditions. Here $T$ denotes a disk with three marked points on its boundary and $\Sigma\times T$ is equipped with an almost complex structure satisfying conditions ($\mathbf{J'1}$)--($\mathbf{J'4}$) in \cite[\textsection 10.2 p. 1018]{Lipshitz}.  

No matter the almost complex structure, the differential on $\cfhat(\Sigma,\bsb',\bsb,\spc_\circ)$ vanishes identically. Therefore, restricting to the subcomplex $\F\cdot \vtheta\otimes_{\F} \cfhat(\Mba,\spc_\xi)$, where $\bstheta=\{\theta_1,\dots,\theta_\N\}$ and $\vtheta$ is the top degree generator of $\cfhat(\Sigma,\bsb',\bsb,\spc_\circ)$, results in a chain map 
\[\hat{f}_{\bsb',\bsb,\bsa;\mathfrak{t}_\xi}(\vtheta\otimes\cdot):\cfhat(\Mba,\spc_\xi)\to \cfhat(\Mbpa,\spc_\xi).\]
The latter induces an isomorphism of homologies by \cite[Proposition 11.4]{Lipshitz} (cf. \cite[Proposition 9.8]{OzsvathSzabo}). In what follows, we work with a generic split complex structure on $\Sigma\times T$. We are allowed to do so since transversality of moduli spaces as defined by such almost complex structures can be guaranteed by slight perturbation of the $\bsa$, $\bsb$, and $\bsb'$ curves. To be more precise, we may invoke the technique of \cite{Oh}. This is because any class $A$ in $\pitwo(\vtheta,\cdot,\cdot)$ satisfies the boundary injectivity criterion in the sense of \cite{Lipshitz}. By way of a reminder, a class $A$ in $\pitwo(\vtheta,\cdot,\cdot)$ is said to satisfy the boundary injectivity criterion if any pseudo-holomorphic curve $u$ for some split complex structure on $\Sigma\times T$ representing the class $A$ has $\pi_\Sigma\circ u$ somewhere injective in its boundary. This criterion is guaranteed as long as the domain representing the class has a region with multiplicity one adjacent to a region of multiplicity zero. Note that this is the case for any class in $\pitwo(\vtheta,\cdot,\cdot)$ due to the placement of the basepoints, in that basepoints appear on both sides of every $\bsa$, $\bsb$, and $\bsb'$ curve.

Next we show that the chain map $\hat{f}_{\bsb',\bsb,\bsa;\mathfrak{t}_\xi}(\vtheta\otimes\cdot)$ induces a morphism of spectral sequences~from $E^\ast(S,\phi,\arcs;J^{}_{\mathit{HF}})$ to $E^\ast(S,\phi',\arcs;J'_{\mathit{HF}})$. First, define an analog of the formula \eqref{eq:jplus-prod} for the cobordism described by the Heegaard triple diagram $(\Sigma,\bsb',\bsb,\bsa)$ via
\begin{equation}\label{eq:jplus-non-prod}
J_+(A)=\frac{\N}{2}+\mu(\DA)-2e(A)+|\mathbf{x}|-|\mathbf{y}|,
\end{equation}
where $A\in \pitwo(\vtheta,\vx,\vy)$, $\mu(\DA)$ denotes the Maslov index of the domain $\DA$ associated to $A$, which is the expected dimension of the moduli space of pseudo-holomorphic curves representing the class $A$, and $e(A)$ is the Euler measure of the domain associated to the class $A$. If $A$ can be represented by an embedded Fredholm index-$0$ pseudo-holomorphic curve $C_L$, then the formula \eqref{eq:jplus-non-prod} becomes
\begin{eqnarray*}
J_+(A)&=&\frac{\N}{2}-2e(A)+|\mathbf{x}|-|\mathbf{y}|\\
&=&\underbrace{-\chi(C_L)+\N}_{\text{by \cite[\textsection 10.2]{Lipshitz}}}+|\mathbf{x}|-|\mathbf{y}|.\\ 
\end{eqnarray*}
It follows from this formula that $J_+(A)=2\ell$ for some $\ell\geq0$. To see this, consider the smooth compact oriented surface $C$ obtained from the compactification of $C_L$ by first adding $2$-dimensional $1$-handles, one for each pair $(\beta'_i,\beta^{}_i)$ and one for each pair $(\beta'_i,\alpha^{}_i)$, and then capping off the boundary components of the resulting surface containing the $I$-chord collection $\vtheta$. Note that $\chi(C)=\chi(C_L)-\N$, and $|\x|$ (resp. $|\y|$) is equal to the number of boundary components of $C$ arising from the $I$-chord collection $\vx$ (resp. $\vy$). The claim then follows in exactly the same way as in \textsection\ref{sec:definition}. Consequently, we can decompose the chain map $\hat{f}_{\bsb',\bsb,\bsa;\mathfrak{t}_\xi}(\vtheta\otimes\cdot)$ as
\[\hat{f}_{\bsb',\bsb,\bsa;\mathfrak{t}_\xi}(\vtheta\otimes\cdot)=f^0+f^1+\cdots+f^\ell+\cdots,\]
where $f^\ell$ counts embedded Fredholm index-$0$ pseudo-holomorphic curves with $J_+=2\ell$. Since the Maslov index and the Euler measure are additive under concatenation, it follows using the formulae \eqref{eq:jplus-prod} and \eqref{eq:jplus-non-prod} that $J_+$ is also additive. Therefore, we have 
\begin{equation}
\label{eq:identity}
\sum_{i+j=\ell}(f^i\circ\partial^{}_{j}-\partial^{'}_{i}\circ f^j)=0
\end{equation}
since $\hat{f}_{\bsb',\bsb,\bsa;\mathfrak{t}_\xi}$ is a chain map and the $J_+$-filtered differential on $\cfhat(\Sigma,\bsb',\bsb,\spc_\circ)$ is identically zero. The latter is due to the fact that $\cfhat(\Sigma,\bsb',\bsb,\spc_\circ)$ is isomorphic to $(\F_{(0)}\oplus\F_{(1)})^{\otimes \N}$ where $\F_{(0)}\oplus\F_{(1)}$ is a graded module over $\F$ with vanishing differential and the domains corresponding to the pseudo-holomorphic curves that contribute to the differential of the generator, $\theta_i\times[0,1]$, of $\F_{(1)}$ are both bigons, which have $J_+=0$. In short, the restriction of the differential on $\cfhat(\Sigma,\bsb',\bsb,\spc_\circ)\otimes_{\F} \cfhat(\Mba,\spc_\xi)$ to the subcomplex $\F\cdot \vtheta\otimes_{\F} \cfhat(\Mba,\spc_\xi)$ is $J_+$-filtered. 

The identity \eqref{eq:identity} implies that there is a filtered chain map from $(\CCF(S,\phi,\arcs),\widehat{\partial})$ to $(\CCF(S,\phi',\arcs),\widehat{\partial}')$ as before, hence a morphism of spectral sequences from $E^\ast(S,\phi,\arcs;J^{}_{\mathit{HF}})$ to $E^\ast(S,\phi',\arcs;J'_{\mathit{HF}})$. In addition, \[\hat{f}_{\bsb',\bsb,\bsa;\mathfrak{t}_\xi}(\vtheta\otimes\vx_\xi)=\vx'_\xi\] since the shaded triangles in Figure \ref{fig:isotopy} constitute the only holomorphic domain that contributes to this chain map due to the placement of the basepoints, and it is represented by a unique pseudo-holomorphic curve by the Riemann Mapping Theorem. Hence, $\ordn(S,\phi,\arcs;J^{}_{\mathit{HF}})\geq \ordn(S,\phi',\arcs;J'_{\mathit{HF}})$. Likewise, the isotopy from $\phi'$ to $\phi$ yields $\ordn(S,\phi,\arcs;J^{}_{\mathit{HF}})\leq \ordn(S,\phi',\arcs;J'_{\mathit{HF}})$. As a result, $\ordn(S,\phi,\arcs;J^{}_{\mathit{HF}})=\ordn(S,\phi',\arcs;J'_{\mathit{HF}})$.
\end{proof}

\begin{rmk} Sarkar--Wang \cite{SarkarWang} and Plamenevskaya \cite{Plamenevskaya} proved that the Heegaard diagram resulting from an arbitrary choice of $(S,\phi,\arcs)$, where $\arcs$ contains a basis, can be made \emph{nice} by choosing $\phi$ appropriately in its isotopy class. On a nice Heegaard diagram, every Maslov index-$1$ holomorphic domain is represented by an \emph{empty} embedded bigon or an \emph{empty} embedded square \cite[Theorem 3.3]{SarkarWang}. It is easy to see from \eqref{eq:jplus-prod} that such domains have either $J_+=0$ or $J_+=2$. This observation indicates that there should be a combinatorial description of $\ordn$. 
\end{rmk}

\subsection{Eliminating triangles}
\label{sec:arcind}
In this subsection, we investigate dependence of $\ordn$ on a choice of pairwise disjoint properly embedded collection of arcs containing a basis. More specifically, given an open book decomposition $(\Pg,\phi)$ and such an arc collection $\arcs$ on $\Pg$, we prove that $\ordn$ is non-increasing under a \emph{triangle elimination} operation on $\arcs$, which we will describe momentarily. As we shall see in Section \ref{sec:proof}, this operation gives us quite a bit of flexibility in our arguments that lead to the proofs of our main theorems. To set the stage, let $\obd$ be an open book decomposition supporting a contact structure $\xi$, and $\arcs=\{a^{}_0,a^{}_3,a^{}_4,\dots,a^{}_\N\}$ be a collection of pairwise disjoint properly embedded arcs on $\Pg$ that contains a basis. Suppose that the three arcs $a^{}_0,a^{}_3,a^{}_4\in\arcs$ bound a connected component of $\Pg\smallsetminus\bigcup\arcs$. Denote by $\arcs'$ the collection of pairwise disjoint properly embedded arcs on $\Pg$ obtained by discarding $a_0$ and ``doubling'' $a^{}_3$ and $a^{}_4$, i.e. $\arcs'=\{a'_1,a'_2,a^{}_3,a^{}_4,\dots,a^{}_\N\}$ where $a'_1$ and $a'_2$ are parallel and sufficiently close to $a^{}_3$ and $a^{}_4$, respectively (see Figure \ref{fig:tc_triangle}). Then:
\vfill
\begin{figure}[h]
\centering
\begin{subfigure}[b]{0.45\linewidth}
\centering
\includegraphics[width=1.85in]{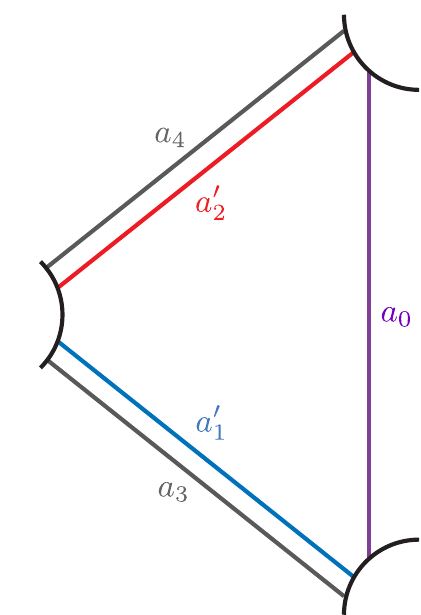}
\caption{}
\label{fig:tc_triangle}
\end{subfigure}%
\begin{subfigure}[b]{0.45\linewidth}
\centering
\includegraphics[width=1.25in]{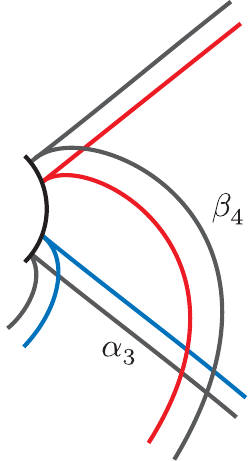}
\caption{}
\label{fig:tc_aftercollapse}
\end{subfigure}
\caption{(a) The configuration of arcs in the $\Pg\times\{\frac{1}{2}\}$ page of the open book decomposition representing a triangle collapse operation. (b) The local behavior of the $\bsb$-curves as shown in the $\Pg\times \{0\}$ half of the Heegaard diagram $(\Sigma,\bsb',\bsa')$.} 
\label{fig:tc}
\end{figure}

\begin{prop}
\label{prop:triangle_collapse}
Let $\obd$ be an open book decomposition, $\arcs$ be a collection of pairwise disjoint properly embedded arcs on $\Pg$ that contains a basis, and $\arcs'$ be obtained from $\arcs$ via triangle elimination. Then $\ordn(\Pg,\phi,\arcs')\leq \ordn(\Pg,\phi,\arcs)$.
\end{prop}
\begin{proof}
Throughout, we assume that the monodromy $\phi$ moves arcs $a^{}_0$, $a^{}_3$, and $a^{}_4$ to the right, since otherwise it would not move $a'_1$ and $a'_4$ to the right either. As a result, $\ordn(\Pg,\phi,\arcs)$ and $\ordn(\Pg,\phi,\arcs')$ would both be zero as in the proof of Theorem \ref{thm:vanishing}. We further assume that $\beta^{}_4$ stays parallel to the boundary of $\Pg$ immediately after turning right in the $\Pg\times \{0\}$ half of the Heegaard diagram until it intersects $\alpha^{}_3$ as in Figure \ref{fig:tc_aftercollapse}. Otherwise (see Figure \ref{fig:non-bounparallel}), isotope the monodromy $\phi$ so as to guarantee that this is the case (see Figure \ref{fig:non-bounparallel_iso}). With this understood, we proceed with the rest of the proof.

\begin{figure}[h]
\centering
\begin{subfigure}[b]{0.45\linewidth}
\centering
\includegraphics[width=1.85in]{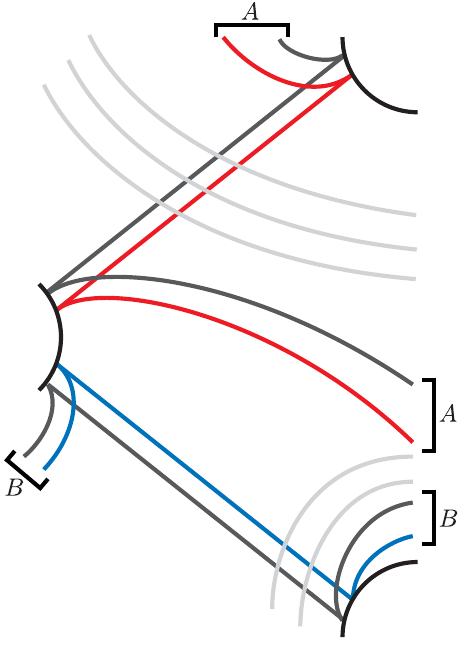}
\caption{}
\label{fig:non-bounparallel}
\end{subfigure}%
\begin{subfigure}[b]{0.45\linewidth}
\centering
\includegraphics[width=1.85in]{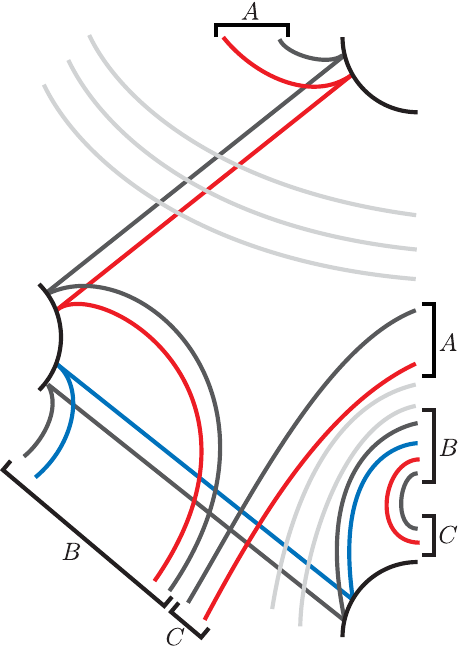}
\caption{}
\label{fig:non-bounparallel_iso}
\end{subfigure}
\caption{(a) Configuration of arcs when $\beta^{}_4$ doesn't stay parallel to the boundary of $\Pg$ immediately after turning right in the $\Pg\times \{0\}$ half of the Heegaard diagram. (b) Configuration of arcs after an isotopy to guarantee that $\beta^{}_4$ intersects $\alpha^{}_3$ immediately after turning right. In both figures, brackets indicate the ends of arcs that are identified.}
\end{figure}

To start, we may isotope $\phi$ so that the multi-pointed Heegaard diagram $(\Sigma,\bsb',\bsa',\base')$ corresponding to $(\Pg,\phi,\arcs')$ is \emph{nice} while making sure that the intersection pattern as depicted in Figure \ref{fig:tc_aftercollapse} is preserved. To be more explicit, starting with the Heegaard diagram resulting from $\phi$, we may apply the algorithm of Sarkar--Wang \cite[\textsection 4.1]{SarkarWang} to produce a nice Heegaard diagram by performing finger moves on $\bsb$-curves only in the $\Pg\times\{0\}$ half of the Heegaard surface $\Sigma$, as is argued in \cite{Plamenevskaya}. This is because, in a Heegaard diagram arising from an open book decomposition, there are regions with basepoint on either side of every $\bsb$-curve. Furthermore, we will show that these finger moves on $\bsb$-curves can be performed in such a way that the arc $\delta$ along $\beta^{}_4$ between the points $x^{}_4$ and $v$, shown in Figure \ref{fig:forbidden}, remains unchanged, and in the resulting nice Heegaard diagram, no $\bsb$-curve forms a bigon with the arc $\eta$ along $\alpha^{}_3$ between the points $x^{}_3$ and $v$. It suffices to perform these finger moves in the Heegaard diagram resulting from the arc collection $\{a^{}_3,a^{}_4,\dots,a^{}_\N\}$ since adding $a'_1$ and $a'_2$, which are parallel to $a^{}_3$ and $a^{}_4$, respectively, merely subdivides bigon and rectangle regions into smaller bigon and rectangle regions. With the preceding understood, we produce a nice diagram with the desired properties in the three steps that follow. Throughout, the \emph{distance} of a region is defined to be the minimum number of intersection points between the $\bsb$-curves and an arc connecting the interior of that region to a region with basepoint in the complement of the $\bsa$-curves and the arc $\delta$. 

\begin{stp}
Note that, given a region, there is exactly one region with basepoint that can be connected to the interior of that region via an arc in the complement of the $\bsa$-curves and the arc $\delta$. Proceed as in the algorithm of Sarkar-Wang by first killing all non-disk regions without performing a finger move starting at $\delta$ and then performing finger moves as in the proof of \cite[Lemma 4.1]{SarkarWang} to reduce the \emph{distance d complexity} of the Heegaard diagram to (0) starting with the largest distance \emph{bad regions}. Note that, the region without a basepoint that has the arc $\delta$ on its boundary is adjacent to a region with distance one less along an arc on a $\bsb$-curve other than $\delta$. Therefore, at no point in the process, finger moves needed to break up the region without a basepoint that has the arc $\delta$ on its boundary into rectangles and bigons start at $\delta$. Continue performing finger moves as in the proof of \cite[Lemma 4.1]{SarkarWang} until the \emph{distance} of the Heegaard diagram is reduced to~$1$, that is, until all bad regions are of distance at most~$1$.
\end{stp}

\begin{figure}[h]
\centering
\begin{subfigure}[b]{0.45\linewidth}
\centering
\includegraphics[width=1.35in]{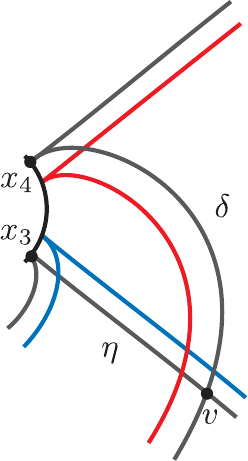}
\caption{}
\label{fig:forbidden}
\end{subfigure}
\begin{subfigure}[b]{0.45\linewidth}
\centering
\includegraphics[width=1.85in]{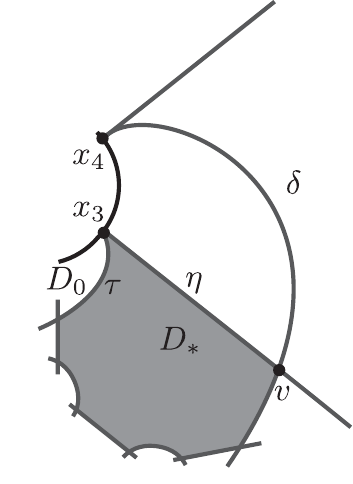}
\caption{}
\label{fig:dstar}
\end{subfigure}
\caption{(a) The arc configuration in the desired nice Heegaard diagram with the arcs prohibited from forming bigons indicated. (b) The domain $D_\ast$, the adjacent region with basepoint $D_0$, and the arc $\tau$ along $\beta^{}_3$ which they both have on their boundaries.}
\label{fig:}
\end{figure}

\begin{stp}
The region with no basepoints and the arc $\eta$ on its boundary has distance $1$, and it is adjacent to a region with basepoint $D_0$ along $\beta^{}_3$. We denote this region by $D_\ast$ (see Figure \ref{fig:dstar}). Continue performing finger moves as in the proof of \cite[Lemma 4.1]{SarkarWang} so as to reduce distance $1$ complexity of the Heegaard diagram. In doing so, all finger moves that may enter $D_\ast$ will be stopped once they enter $D_\ast$. This is contrary to what the Sarkar--Wang algorithm requires: a finger pushed into an equal or larger distance rectangle should exit the rectangle on the opposite edge. Our goal is to break all distance $1$ bad regions up into rectangles and bigons except possibly $D_\ast$. The reason why this is possible is because the Sarkar--Wang algorithm terminates after a finite number of finger moves, and we can stop those finger moves that enter $D_\ast$ once they enter $D_\ast$ while using a modified version of distance $1$ complexity of the Heegaard diagram that disregards the region $D_\ast$ at every step of the process. This modification of the algorithm does not increase distance of any bad regions, and the modified algorithm eventually breaks every bad region other than possibly $D_\ast$ into rectangles and bigons at the expense of possibly increasing the badness of $D_\ast$. 
\end{stp}

\begin{figure}[h]
\centering
\begin{subfigure}[b]{0.45\linewidth}
\centering
\includegraphics[width=1.85in]{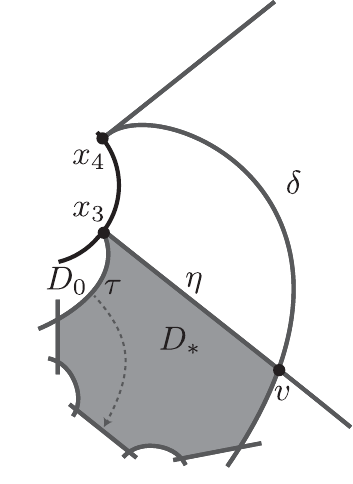}
\caption{}
\label{fig:dstarfinger}
\end{subfigure}
\begin{subfigure}[b]{0.45\linewidth}
\centering
\includegraphics[width=1.85in]{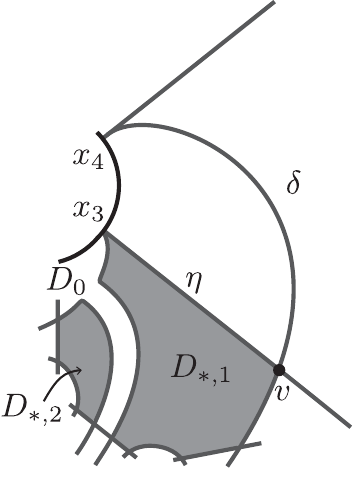}
\caption{}
\label{fig:dstarbreak}
\end{subfigure}
\caption{(a) The dashed line indicates the finger move to break up $D_\ast$. (b) The regions $D_{\ast,1}$ and $D_{\ast,2}$ formed after the finger move.}
\label{fig:algorithm}
\end{figure}

\begin{stp}
Finally, we break up the only remaining bad region, namely, $D_\ast$. We claim that we can perform a sequence of finger moves as in the proof of \cite[Lemma~4.1]{SarkarWang} so that, in the resulting nice Heegaard diagram, no $\bsb$-curve forms a bigon with $\eta$. We prove this claim by strong induction on the badness $b(D_\ast)$ of the region $D_\ast$. By way of reminder, the badness of a $2n$-gon is defined in \cite[\textsection 4.1]{SarkarWang} to be $\max\{n-2,0\}$. If $b(D_\ast)=1$, that is, if $D_\ast$ is a hexagon, then performing a finger move as in the proof of \cite[Lemma 4.1]{SarkarWang} starting at the arc $\tau$ along $\beta^{}_3$ with an end at $x^{}_3$ on the boundary of $D_\ast$ breaks $D_\ast$ up into two rectangles. Moreover, this finger won't come back to $D_\ast$ since otherwise it would have to follow a full $\bsb$-curve, which in turn would force our finger to cross a region with basepoint because there are regions with basepoint on either side of every $\bsb$-curve. Next suppose that $b(D_\ast)>1$ and perform a finger move as in the proof of \cite[Lemma 4.1]{SarkarWang} starting at $\tau$ (see Figure \ref{fig:dstarfinger}). If the finger doesn't come back to $D_\ast$, then it will end up in a bigon region or a region with basepoint, and $D_\ast$ will be broken up into a region $D_{\ast,1}$ with badness $b(D_\ast)-1$ and a rectangle $D_{\ast,2}$. Note that both $D_{\ast,1}$ and $D_{\ast,2}$ are adjacent to $D_0$ along $\beta^{}_3$, and that $D_{\ast,1}$ has the arc $\eta$ on its boundary (see Figure \ref{fig:dstarbreak}). Then by the induction hypothesis, the claim is true. Suppose instead that the finger comes back to $D_\ast$. Then, by the argument in \cite[Subcase 4.2]{SarkarWang} and the fact that there are regions with basepoint on either side of every $\bsb$-curve, there exists another finger move starting at $\tau$ that doesn't come back to $D_\ast$. This finger move would break $D_\ast$ up into two regions $D_{\ast,1}$ and $D_{\ast,2}$ both are adjacent to $D_0$ along $\beta^{}_3$, and $D_{\ast,1}$ has the arc $\eta$ on its boundary. Then we have $b(D_{\ast,1})+b(D_{\ast,2})=b(D_\ast)-1$, and $b(D_{\ast,2})\geq 1$. Once again, in contrast to the Sarkar--Wang algorithm, which requires ordering bad regions with increasing badness and then breaking up bad regions starting with the regions having the least positive badness, we first break up the region $D_{\ast,2}$ regardless of whether it is a bad region with the least positive badness. As we perform finger moves to break up $D_{\ast,2}$, as well as any subsequent new bad region that might emerge in that process, we stop a finger move that enters $D_{\ast,1}$ once it enters $D_{\ast,1}$, regardless of whether $b(D_{\ast,1})>0$ or not. In order to break up a bad region with badness $b$ into rectangles we need to perform exactly $b$ finger moves, assuming no finger comes back to that region, and each finger pushed into a region would increase its badness by $1$. Therefore, the process of breaking up $D_{\ast,2}$ into rectangles would increase the badness of $D_{\ast,1}$ by at most $b(D_{\ast,2})$. In the end, we have a Heegaard diagram with a single bad region of distance $1$ adjacent to $D_0$ along $\tau$ having the arc $\eta$ on its boundary. The badness of this region is at most $b(D_{\ast,1})+b(D_{\ast,2})=b(D_\ast)-1$. Hence, by the induction hypothesis, our claim holds true, and a further sequence of finger moves as described above yields the desired nice Heegaard diagram.
\end{stp}

Next, add an isotopic copy of $a^{}_0$ in such a way that $\alpha_0$ and $\beta_0$ intersect $\beta_4$ and $\alpha_3$, respectively, to form bigons as in Figure \ref{fig:tc_beforearcs}. Then the multi-pointed Heegaard diagram $(\Sigma,\bsb,\bsa,\base)$ corresponding to $(\Pg,\phi,\arcs)$ is also nice. This is because there is a canonical 1--1 correspondence between the regions without a basepoint in the Heegaard diagram $(\Sigma,\bsb',\bsa',\base')$ and the regions without a basepoint in the Heegaard diagram $(\Sigma,\bsb,\bsa,\base)$ in the complement of the shaded area in Figure \ref{fig:tc_regions}. In the shaded area, all regions without basepoints belonging to the Heegaard diagram $(\Sigma,\bsb,\bsa,\base)$ are easily seen to be bigons. Since, by Proposition 3.2, $\ordn(\Pg,\phi,\arcs')$ does not change under isotopy of $\phi$, we may assume throughout the rest of the proof that the monodromy $\phi$ is such that the multi-pointed Heegaard diagrams $(\Sigma,\bsb',\bsa',\base')$ and $(\Sigma,\bsb,\bsa,\base)$ are both nice and have the intersection patterns depicted in Figures \ref{fig:tc_aftercollapse} and \ref{fig:tc_beforearcs}.

\begin{figure}[h]
\centering
\begin{subfigure}[b]{0.45\linewidth}
\centering
\includegraphics[width=1.25in]{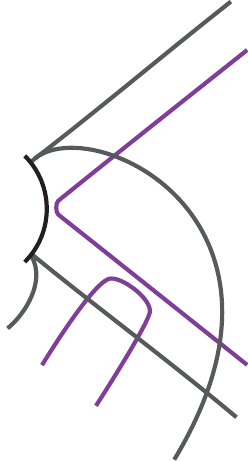}
\caption{}
\label{fig:tc_beforearcs}
\end{subfigure}
\begin{subfigure}[b]{0.45\linewidth}
\centering
\includegraphics[width=1.25in]{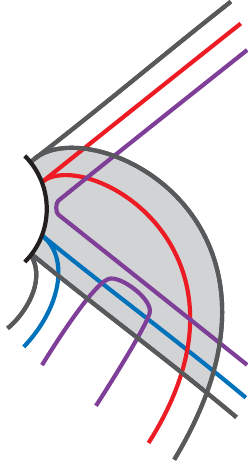}
\caption{}
\label{fig:tc_regions}
\end{subfigure}
\caption{(a) The $\Pg\times \{0\}$ half of the multi-pointed Heegaard diagram $(\Sigma,\bsb,\bsa,\base)$. (b) The shaded area in which the regions in the two multi-pointed Heegaard diagrams essentially differ.}
\label{tc_beforecollapse}
\end{figure}

With the preceding understood, to each ($\N-1$)-tuple of intersection points $\y=\{y_0,y_3,y_4,\dots,y_\N\}$ defining a generator $\vy$ of $\cfhat(\Sigma,\bsb,\bsa)$ associate a unique ($\N-1$)-tuple of intersection points $\y'$ in $\bsa'\cap\bsb'$ using the following recipe. For points belonging to $\y$ that lie on $\alpha_0$ or $\beta_0$, associate a unique point in $\bsa'\cap\bsb'$ according to the following rules:
\begin{itemize}\leftskip-0.25in
\item If $y_0\in\alpha_0\cap\beta_0$ and $y_0\neq x_0$, then the associated point in $\bsa'\cap\bsb'$ lies in $\alpha'_i\cap\beta'_j$ where $i,j\in\{1,2\}$ (see Figure \ref{fig:tc_case1}). If $y_0=x_0$, we associate to it the point $x'_1$. 
\item If $y_0\in\alpha_0\cap\beta_j$ where $j\geq3$, then the associated point in $\bsa'\cap\bsb'$ lies in $\alpha'_i\cap\beta'_j$ where $i\in\{1,2\}$ (see Figure \ref{fig:tc_case2}).
\item If $y_i\in\alpha_i\cap\beta_0$ where $i\geq3$, then the associated point in $\bsa'\cap\bsb'$ lies in $\alpha'_i\cap\beta'_j$ where $j\in\{1,2\}$ (see Figure \ref{fig:tc_case3}).
\end{itemize}
\begin{figure}[h]
\centering
\begin{subfigure}[b]{0.3\linewidth}
\centering
\includegraphics[width=1.25in]{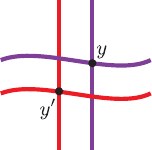}
\caption{}
\label{fig:tc_case1}
\end{subfigure}%
\begin{subfigure}[b]{0.3\linewidth}
\centering
\includegraphics[width=1.25in]{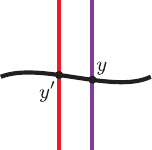}
\caption{}
\label{fig:tc_case2}
\end{subfigure}
\begin{subfigure}[b]{0.3\linewidth}
\centering
\includegraphics[width=1.25in]{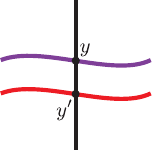}
\caption{}
\label{fig:tc_case3}
\end{subfigure}
\caption{Assigning to an intersection point in $\bsa\cap\bsb$ and intersection point in $\bsa'\cap\bsb'$. Straight arcs indicate $\bsa$-curves, while wavy arcs indicate $\bsb$-curves. Purple color corresponds to $\alpha^{}_0$ or $\beta^{}_0$, black color corresponds to $\alpha^{}_i$ or $\beta^{}_j$ for $i,j\geq3$, and red color corresponds to $\alpha'_i$ or $\beta'_j$ for $i,j\in\{1,2\}$.}
\label{fig:tc_cases}
\end{figure}
In all other cases, the intersection points remain the same. Note that $\y'$ uses exactly one of $\alpha'_1$ or $\alpha'_2$, and exactly one of $\beta'_1$ and $\beta'_2$. Depending on which pair of $\alpha'_i$ and $\beta'_j$ that $\y'$ uses, we assign $\y$ the ordered pair $p_\y:=(i,j)$. Then, unless $p_\y=(1,2)$, we associate to $\y$ a unique $\N$-tuple of intersection points $\tilde{\y}:=\{y'_1,y'_2,y'_3,y'_4,\dots,y'_\N\}$ defining a generator $\vty$ of the chain complex $\cfhat(\Sigma,\bsb',\bsa')$ by adding to $\y'$ 
\begin{itemize}\leftskip-0.25in
\item the point $x'_2$ if $p_\y=(1,1)$,
\item the point $w\in\alpha'_2\cap\beta'_1$ indicated in Figure \ref{fig:tc_w} if $p_\y=(2,1)$,
\item the point $x'_1$ if $p_\y=(2,2)$.
\end{itemize}
\begin{figure}[h]
\centering
\includegraphics[width=1.25in]{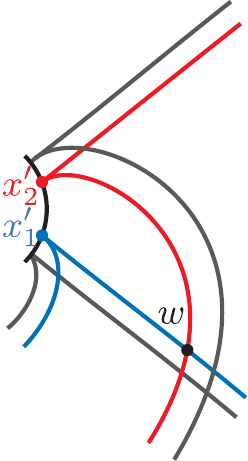}
\caption{The intersection point $w$.}
\label{fig:tc_w}
\end{figure}
Note that this recipe associates to the distinguished ($\N-1$)-tuple of intersection points $\x^{}_\xi=\{x^{}_0,x^{}_3,x^{}_4,\dots,x^{}_\N\}$ the distinguished $\N$-tuple of intersection points $\x'_\xi=\{x'_1,x'_2,x^{}_3,x^{}_4,\dots,x^{}_\N\}$. These two sets of intersection points define the distinguished generators that represent the Ozsv\'ath--Szab\'o contact class in the homology of the chain complexes $\cfhat(\Sigma,\bsb,\bsa)$ and $\cfhat(\Sigma,\bsb',\bsa')$, respectively.
\vfill
\begin{lem}
\label{lem:py_preserved}
Let $\D\in\pitwo(\vy^1,\vy^2)$ be a Maslov index-1 holomorphic domain. Then $p_{\y^1}=p_{\y^2}$ unless $\D$ is a bigon. Furthermore, if $p_{\y^1}=(1,2)$, then $p_{\y^2}=(1,2)$.
\end{lem}
\begin{proof}
Given $\y^1=\{y^1_0,y^1_3,y^1_4,\dots,y^1_\N\}$ defining a generator $\vy^1$ of $\cfhat(\Sigma,\bsb,\bsa)$, the first entry of the ordered pair $p_{\y^1}$ is determined by $y^1_0$, specifically by whether $y^1_0$ is near $\alpha'_1$ or $\alpha'_2$. Similarly, the second entry of $p_\y$ is determined by $y^1_i\in \alpha^{}_i\cap\beta^{}_0$, specifically by whether $y^1_i$ is near $\beta'_1$ or $\beta'_2$. Let $\D\in\pitwo(\vy^1,\vy^2)$ be a Maslov index-1 holomorphic domain that is a rectangle having both $\alpha_0$ and $\beta_0$ on its boundary (the assertion of the lemma is clear for other types of rectangular domains). If $p_{\y^1}=(i,j)$, then $\D$ has to have an edge parallel to $\alpha_3$ or to $\alpha_4$ depending on whether $i=1$ or $i=2$, and an edge parallel to $\beta_3$ or $\beta_4$ depending on whether $j=1$ or $j=2$. This is because the only Maslov index-1 rectangular domains that would not have this property would be ones that have an edge along $\alpha^{}_0$ or along $\beta^{}_0$ running parallel to both $\alpha_3$ and $\alpha_4$ or both $\beta_3$ and $\beta_4$, respectively, on its boundary. But there are no such rectangular domains in the nice Heegaard diagrams we produced above since $\delta$ does not form a bigon with any $\bsa$-curve other than $\alpha_0$ and no $\bsb$-curve other than $\beta_0$ forms a bigon with $\eta$. As a result, $p_{\y^2}=(i,j)$. On the other hand, if $\D$ is a bigon and $p_{\y^1}\neq p_{\y^2}$, then we have either $p_{\y^1}=(2,j)$ or $p_{\y^1}=(i,1)$ while $p_{\y^2}=(1,j)$ or $p_{\y^2}=(i,2)$, respectively. It follows, in particular, that if $p_{\y^1}=(1,2)$, then $p_{\y^2}=(1,2)$.
\end{proof}
Consequently, the submodule of $\cfhat(\Sigma,\bsb,\bsa)$ generated by $\vy$ with $p_\y=(1,2)$ is a subcomplex. We will denote this subcomplex by $\cfhat_\circ(\Sigma,\bsb,\bsa)$ for future reference. Next we investigate the image under the differential of a generator $\vty^1$ of $\cfhat(\Sigma,\bsb',\bsa')$  corresponding to a generator $\vy^1$ of $\cfhat(\Sigma,\bsb,\bsa)$.
\begin{lem}
\label{lem:moduli_(12)}
If $p_{\y^1}\neq(1,2)$, then there exists a Maslov index-1 holomorphic domain $\D'\in\pitwo(\vty^1,\vty)$ only if $\tilde{\y}=\tilde{\y}^2$ for some generator $\vy^2$ of $\cfhat(\Sigma,\bsb,\bsa)$ with $p_{\y^2}\neq(1,2)$.
\end{lem}
\begin{proof}
To see this, write $\tilde{\y}^1=\{{y'}^1_1,{y'}^1_2,{y'}^1_3,{y'}^1_4,\dots,{y'}^1_\N\}$, $\tilde{\y}=\{{y'}^{}_1,{y'}^{}_2,{y'}^{}_3,{y'}^{}_4,\dots,{y'}^{}_\N\}$, and recall that either ${y'}^1_1={x'}^{}_1$, ${y'}^1_1=w$, or ${y'}^1_2={x'}^{}_2$. If ${y'}^1_1={x'}^{}_1$ or ${y'}^1_2={x'}^{}_2$, then ${y'}^{}_1={x'}^{}_1$ or ${y'}^{}_2={x'}^{}_2$, respectively, since there are no non-trivial Maslov index-1 holomorphic domains with a corner at ${x'}^{}_1$ or ${x'}^{}_2$. If ${y'}^1_1=w$, then either ${y'}^{}_1={x'}^{}_1$, ${y'}^{}_1=w$ or ${y'}^{}_2={x'}^{}_2$ since a Maslov index-1 holomorphic domain with a corner at $w$ has to have a corner at ${x'}^{}_1$ or ${x'}^{}_2$. The latter is due to the fact that the multi-pointed Heegaard diagram $(\Sigma,\bsb',\bsa',\base')$ is nice, hence all Maslov index-1 holomorphic domains are empty embedded bigons or rectangles, and that starting at $w$ and moving along $\alpha'_1$ or $\beta'_2$ there is nowhere else to turn a corner other than at $x'_1$ or at $x'_2$. As a result, $\tilde{\y}=\tilde{\y}^2$ for some generator $\vy^2$ of $\cfhat(\Sigma,\bsb,\bsa)$ with $p_{\y^2}\neq(1,2)$.
\end{proof}
\begin{lem}
\label{lem:moduli_1to1}
If $p_{\y^1}\neq(1,2)$ and $p_{\y^2}\neq(1,2)$, then there is a canonical 1--1 correspondence between Maslov index-1 holomorphic domains in $\pitwo(\vy^1,\vy^2)$ and Maslov index-1 holomorphic domains in $\pitwo(\vty^1,\vty^2)$. 
\end{lem}
\begin{proof}
Keep in mind that the Heegaard diagrams $(\Sigma,\bsb,\bsa,z)$ and $(\Sigma,\bsb',\bsa',z')$ are both nice. In particular, a Maslov index-1 holomorphic domain has a unique holomorphic representative up to translation. If $\vy^1$ and $\vy^2$ are generators of $\cfhat(\Sigma,\bsb,\bsa)$ with $p_{\y^1}\neq(1,2)$ and $p_{\y^2}\neq(1,2)$, then a Maslov index-1 holomorphic domain $\D\in\pitwo(\vy^1,\vy^2)$ gives rise to a canonical Maslov index-1 holomorphic domain $\D'\in\pitwo(\vty^1,\vty^2)$, and vice versa. If a domain $\D$ has neither $\alpha_0$ nor $\beta_0$ on its boundary, then $\D'=\D$. Otherwise, to construct $\D'$ from $\D$ we add rectangular regions between $\alpha_0$ and $\alpha'_1$, $\alpha_0$ and $\alpha'_2$, $\beta_0$ and $\beta'_1$ or $\beta_0$ and $\beta'_2$, while removing the bigon regions between $\alpha_0$ and $\beta'_2$ or $\alpha'_1$ and $\beta_0$ as needed (see Figure \ref{fig:collapse_domains}). The former operation is reversible if $\D'$ has $\alpha'_1$ or $\alpha'_2$, and $\beta'_1$ or $\beta'_2$ on its boundary.
\end{proof}

\begin{figure}[h]
\centering
\begin{subfigure}[b]{0.3\linewidth}
\centering
\includegraphics[width=1.25in]{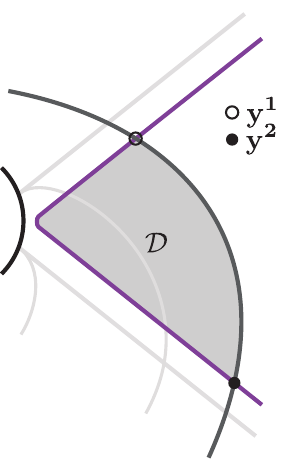}
\caption{}
\label{fig:before_collapse}
\end{subfigure}%
\begin{subfigure}[b]{0.3\linewidth}
\centering
\includegraphics[width=1.25in]{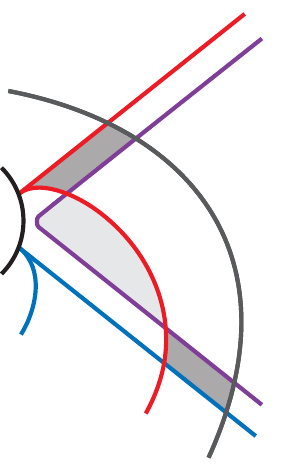}
\caption{}
\label{fig:regions}
\end{subfigure}
\begin{subfigure}[b]{0.3\linewidth}
\centering
\includegraphics[width=1.25in]{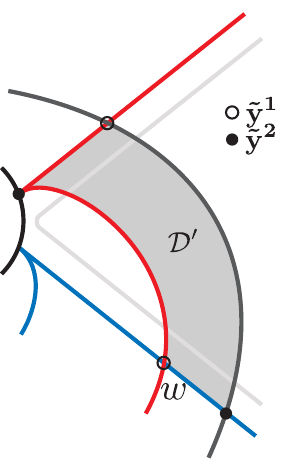}
\caption{}
\label{fig:after_collapse}
\end{subfigure}
\caption{Constructing domains in $(\Sigma,\bsb',\bsa',\base')$ from domains in $(\Sigma,\bsb,\bsa,\base)$. Start with a domain in the multi-pointed Heegaard diagram $(\Sigma,\bsb,\bsa,\base)$ as in (a), add the darker shaded rectangular regions and subtract the lighter shaded bigon region in (b) to get the domain in the multi-pointed Heegaard diagram $(\Sigma,\bsb',\bsa',\base')$ shown in (c).}
\label{fig:collapse_domains}
\end{figure}

It is useful to note here that if $\vy^1$ and $\vy^2$ are generators of $\cfhat(\Sigma,\bsb,\bsa)$ with $p_{\y^1}\neq(1,2)$ and $p_{\y^2}\neq(1,2)$ and $\D\in\pitwo(\vy^1,\vy^2)$ is a Maslov index-1 holomorphic domain, then the corresponding Maslov index-1 holomorphic domain $\D'\in\pitwo(\vty^1,\vty^2)$ has $J_+(\D')=J_+(\D)$. To see this, first note the following:
\begin{itemize}\leftskip-0.25in
\item If $p_\y=(1,1)$ or $p_\y=(2,2)$, then $|\tilde{\y}|=|\y|+1$.
\item If $p_\y=(2,1)$, then $|\tilde{\y}|=|\y|$.
\end{itemize} 
As before, if $\D$ has neither $\alpha_0$ nor $\beta_0$ on its boundary, then $\D'=\D$, hence $J_+(\D')=J_+(\D)$. Now suppose that $\D$ has either $\alpha_0$ or $\beta_0$ on its boundary. 
\begin{itemize}\leftskip-0.25in
\item If $\D$ is a rectangle, then $p_{\y^1}=p_{\y^2}$ (by Lemma \ref{lem:py_preserved}) and $\D'$ is a rectangle. Hence $|\tilde{\y}^1|-|\tilde{\y}^2|=|\y^1|-|\y^2|$ and $J_+(\D')=J_+(\D)$.
\item If $\D$ is a bigon, then $p_{\y^1}=(2,1)$ (otherwise, $p_{\y^2}=(1,2)$) and either $p_{\y^2}=(1,1)$ or $p_{\y^2}=(2,2)$ (by Lemma \ref{lem:py_preserved}), and $\D'$ is a rectangle. Hence $|\tilde{\y}^1|-|\tilde{\y}^2|=|\y^1|-|\y^2|-1$ and 
\[J_+(\D')=2\cdot 1-1+|\tilde{\y}^1|-|\tilde{\y}^2|=2\cdot\frac{1}{2}-1+|\y^1|-|\y^2|=J_+(\D).\]
\end{itemize} 

By Lemma \ref{lem:py_preserved}, the module $\cfhat_\circ(\Sigma,\bsb,\bsa)$ generated by $\vy$ with $p_\y=(1,2)$ is a subcomplex of $\cfhat(\Sigma,\bsb,\bsa)$. Therefore, we may construct the quotient complex $\cfhat(\Sigma,\bsb,\bsa)/\cfhat_\circ(\Sigma,\bsb,\bsa)$. Note that since $p_{\x^{}_\xi}=(1,1)$, it is sent under the quotient map $q:\cfhat(\Sigma,\bsb,\bsa)\to\cfhat(\Sigma,\bsb,\bsa)/\cfhat_\circ(\Sigma,\bsb,\bsa)$ to a non-zero class. The filtered extension of the quotient, $\big(\cfhat(\Sigma,\bsb,\bsa)/\cfhat_\circ(\Sigma,\bsb,\bsa)\big)\otimes_\F \F[t,t^{-1}]$, is canonically isomorphic as a filtered chain complex to the quotient $\CCF(\Pg,\phi,\arcs)/\CCF_\circ(\Pg,\phi,\arcs)$. The quotient map 
\[\CCF(\Pg,\phi,\arcs)\to \CCF(\Pg,\phi,\arcs)/\CCF_\circ(\Pg,\phi,\arcs),\]
is a filtered chain map, and it induces a morphism of associated spectral sequences. Therefore, if we define $\ordn_q(\Pg,\phi,\arcs)$ to be the spectral order as determined by the class $q(\x^{}_\xi)$ and the spectral sequence associated to the filtered quotient chain complex $\big(\cfhat(\Sigma,\bsb,\bsa)/\cfhat_\circ(\Sigma,\bsb,\bsa)\big)\otimes_\F \F[t,t^{-1}]$, then $\ordn(\Pg,\phi,\arcs)\geq\ordn_q(\Pg,\phi,\arcs)$. Meanwhile, by Lemmas \ref{lem:moduli_(12)} and \ref{lem:moduli_1to1}, there exists an injective map from $\cfhat(\Sigma,\bsb,\bsa)/\cfhat_\circ(\Sigma,\bsb,\bsa)$ to $\cfhat(\Pg,\phi,\arcs')$ sending $\vx^{}_\xi$ to $\vx'_\xi$, hence an injective map of filtered chain complexes from $\big(\cfhat(\Sigma,\bsb,\bsa)/\cfhat_\circ(\Sigma,\bsb,\bsa)\big)\otimes_\F \F[t,t^{-1}]$ into $\CCF(\Pg,\phi,\arcs')$ which induces a morphism of associated spectral sequences. As a result, $\ordn_q(\Pg,\phi,\arcs)\geq\ordn(\Pg,\phi,\arcs')$, finishing the proof.\qedhere 
\end{proof}
\begin{defn}
It follows from Proposition \ref{prop:triangle_collapse} that for the purpose of defining the contact invariant $\ord$ it suffices to work with arc collections that are bases with multiple parallel copies of some arcs added since one can always pass to such an arc collection, which we will refer to as a \emph{multi-basis}, via triangle elimination without increasing the value of $\ordn$. In other words, we may define $\ord(M,\xi)$ to be the minimum of $\ordn(\Pg,\phi,\arcs)$ over all choices of open book decompositions $(\Pg,\phi)$ of $M$ supporting $\xi$ and multi-bases $\arcs$.
\end{defn}
\vspace{-0.2in}

%%%%%%%%%%%%%%%%%%
\section{Properties of $\ord$}
\label{sec:proof}
%!TEX root = Filtering.tex

The first bullet point of Theorem \ref{thm:main}, that is, $\ord$ vanishes for overtwisted contact structures, was proved at the end of Section \ref{sec:definition}. This section proves the remaining properties of the contact invariant $\ord$ summarized in Theorems \ref{thm:main}, \ref{cor:cobord}, and \ref{thm:consum}.

To start, we establish a few basic properties of $\ordn$. To do so, we work in a slightly more general context where we consider arc collections that may not contain a basis. Let $(S,\phi)$ be an open book decomposition. Given an arc collection $\arcs$ on $S$ that does not necessarily contain a basis, we can extend it to an arc collection $\tilde{\arcs}$ that contains a basis. Then we fix a generic almost complex structure $\acs$ for the multi-pointed Heegaard diagram $(\Sigma,\tilde{\bsb},\tilde{\bsa},\tilde{\base})$ associated to the arc collection $\tilde{\arcs}$. We may regard $\cfhat(\Sigma,\bsb,\bsa)$ as a submodule of $\cfhat(\Sigma,\tilde{\bsb},\tilde{\bsa})$ by identifying the generators of $\cfhat(\Sigma,\bsb,\bsa)$ with the generators obtained from these via adding on the distinguished points lying in $S\times\{\frac{1}{2}\}$ for each of the arcs in $\arcs\smallsetminus\tilde{\arcs}$. Due to the placement of the basepoints there can be no pseudo-holomorphic curves with negative punctures at the chords resulting from these points. Therefore, the differential on $\cfhat(\Sigma,\bsb,\bsa)$ and on the submodule of $\cfhat(\Sigma,\tilde{\bsb},\tilde{\bsa})$ that it is identified with coincide. As a result, we may consider $\cfhat(\Sigma,\bsb,\bsa)$ as a subcomplex of $\cfhat(\Sigma,\tilde{\bsb},\tilde{\bsa})$. With the preceding understood, the first basic property of $\ordn$ is that it is non-increasing under enlargement of arc collections.

\begin{lem}
\label{lem:incl-arcs} 
Suppose that $\arcs_1\subset\arcs_2$ are two collections of pairwise disjoint properly embedded arcs on $S$. Then there exists a generic almost complex structure $\acs$ on $\Sigma\times[0,1]\times\R$, and an inclusion of chain complexes
\[I:\cfhat (\Sigma,\bsb_1,\bsa_1) \to \cfhat (\Sigma,\bsb_2,\bsa_2),\]
\[\mathcal{I}:\CCF (\Pg, \phi, \arcs_1) \to \CCF (\Pg, \phi, \arcs_2),\]
such that the contact generator is mapped to the contact generator by the first inclusion, while the latter inclusion induces a morphism of spectral sequences from $E^\ast(\Pg,\phi,\arcs_1;\acs)$ to $E^\ast(\Pg,\phi,\arcs_2;\acs)$; hence, $\ordn(\Pg, \phi, \arcs_1;\acs)\geq \ordn(\Pg, \phi, \arcs_2;\acs)$.
\end{lem}

\begin{proof}
It suffices to find a generic almost complex structure $\acs$ on $\Sigma\times[0,1]\times\R$ so that moduli spaces of $\acs$-holomorphic curves associated to the Heegaard diagram $(\Sigma,\bsb_2,\bsa_2)$ are cut out transversally, because this immediately implies transversality of moduli spaces of $\acs$-holomorphic curves associated to the Heegaard diagram $(\Sigma,\bsb_1,\bsa_1)$. Having fixed such a generic almost complex structure, the inclusion map $I$ is defined on the set of generators of $\cfhat(\Sigma,\bsb_1,\bsa_1)$ by
\[I(\vy)=\vy'\]
where $\y'=\y\cup\{x_a\}_{a\in\arcs_2\smallsetminus\arcs_1}$ and $x_a$ is the unique intersection point of $a$ and $b$ for an arc $a\in \arcs_2\smallsetminus\arcs_1$. It follows that $I(\vx^1_\xi)=\vx^2_\xi$. Meanwhile, the $\acs$-holomorphic curves that define the differential acting on elements of the subgroup $I(\cfhat (\Sigma,\bsb_1,\bsa_1))$ are the same as the $\acs$-holomorphic curves that define the differential on $\cfhat (\Sigma,\bsb_1,\bsa_1)$. Therefore, $I$ is a chain map, and the induced inclusion map $\mathcal{I}$ is a filtered chain map. The latter induces a morphism of spectral sequences from $E^\ast(\Pg,\phi,\arcs_1;\acs)$ to $E^\ast(\Pg,\phi,\arcs_2;\acs)$; hence, $\ordn(\Pg, \phi, \arcs_1;\acs)\geq \ordn(\Pg, \phi, \arcs_2;\acs)$.
\end{proof}

The next lemma claims that $\ordn$ remains the same under suitable enlargement of the pages of an open book decomposition while keeping the arc collection untouched.

\begin{lem}
\label{lem:incl-surface} 
Let $\arcs$ be a collection of pairwise disjoint properly embedded arcs on $S$, and $\Pg'$ be a compact oriented surface with boundary obtained from $S$ by attaching 1-handles away from a neighborhood of $\partial\arcs$. Let $\phi':\Pg'\to\Pg'$ be an orientation-preserving diffeomorphism whose restriction to $\arcs$ agrees with $\phi$. Then there are generic almost complex structures $\acs^{}$ and $\acs'$ to define the differentials on $\cfhat(\Sigma,\bsb,\bsa)$ and $\cfhat(\Sigma',\bsb,\bsa)$, respectively, such that $(\CCF(\Pg,\phi,\arcs),\widehat{\partial})$ and $(\CCF(\Pg',\phi',\arcs),\widehat{\partial}')$ are isomorphic as filtered chain complexes. As a result, $\ordn(\Pg,\phi,\arcs;\acs^{})=\ordn(\Pg',\phi',\arcs;\acs')$.
\end{lem}

\begin{proof}
It follows from the description of the surface $\Pg'$ that $\arcs$ can also be seen as a pairwise disjoint collection of properly embedded arcs on $\Pg'$. Moreover, there is a canonical 1--1 correspondence between unordered tuples of intersection points in the Heegaard diagrams $(\Sigma,\bsb,\bsa)$ and $(\Sigma',\bsb,\bsa)$. Also note that $\Sigma'$ is obtained from $\Sigma$ by connected summing with tori along regions in the Heegaard diagram $(\Sigma,\bsb,\bsa)$ with basepoints. Therefore, having fixed a generic almost complex structure $\acs^{}$ on $\Sigma\times[0,1]\times\R$, we can ``extend'' it to a generic almost complex structure $\acs'$ on $\Sigma'\times[0,1]\times\R$ so that the holomorphic domains in the pointed Heegaard diagrams $(\Sigma,\bsb,\bsa,\base)$ and $(\Sigma',\bsb,\bsa,\base)$ agree, and the claim follows.
\end{proof}

With the above understood, the proofs of Theorems \ref{thm:main}, \ref{cor:cobord}, and \ref{thm:consum} require working with a more tractable version of $\ord$: 
\begin{defn} 
\label{def:algtor-complete}
Let $(M,\xi)$ be a closed contact $3$-manifold. Fix an open book decomposition $\B=(S,\phi)$ of $M$ supporting $\xi$. Then define
\[\bbord(\B):=\min_{\arcs}\{\ordn(S,\phi,\arcs)\},\]
where the minimum is taken over all choices of multi-basis $\arcs$ on $S$. Indeed, \[\ord(M,\xi)=\min_{\B}\{\bbord(\B)\}.\]
\end{defn}
The quantity $\bbord$ yields an invariant of open book decompositions. We would like to understand its behavior under positive stabilization. Recall that a positive stabilization of an open book decomposition $(S,\phi)$ is an open book decomposition $(S',\phi')$ where $S'$ is obtained from $S$ by attaching a $1$-handle $H$, and $\phi'$ differs from $\phi$ by a right-handed Dehn twist around a simple closed curve $c \subset S'$ that intersects the cocore of $H$ in exactly one point; in other words, $\phi'=\phi\circ\tau_c$. As we will show next, $\bbord$ is non-increasing under positive stabilization. To prove this, we need the flexibility to move from one arc collection to another without increasing the value of $\bbord$. Recall that one can pass from one basis on $S$ to another via a sequence of \emph{arc slide}s. Given a basis $\{a_1,a_2,\dots,a_\G\}$ on $S$ where $a_1$ and $a_2$ are adjacent, namely, there is an arc $\tau\subset\partial S$ with endpoints on $a_1$ and $a_2$ that intersects no other $a_i$, define $a_1+a_2$ to be a properly embedded arc in $S$ isotopic rel $\partial(a_1\cup a_2)\smallsetminus\partial\tau$ to $a_1\cup\tau\cup a_2$ and is disjoint from all other $a_i$. Then passing from $\{a_1,a_2,\dots,a_\G\}$ to $\{a_1+a_2,a_2,\dots,a_\G\}$ is called an arc slide. Somewhat similarly, given a multi-basis on $S$, one can pass to a multi-basis containing an arbitrary arc basis on $S$ via a sequence of \emph{multi-arc slide}s. Given a multi-basis $\arcs$ containing a basis $\{a_1,a_2,\dots,a_\G\}$ on $S$ where $a_1$ and $a_2$ are adjacent and $\arcs$ contains $m$ parallel copies of the arc $a_1$, a multi-arc slide removes all parallel copies of the arc $a_1$ and adds $m+1$ parallel copies of the arc $a_1+a_2$ as well as $m$ additional parallel copies of the arc $a_2$. This modification is equivalent to adding a copy of the arc $a_1+a_2$ and then removing each parallel copy of the arc $a_1$ one by one via triangle elimination, resulting in a new multi-basis $\arcs'$. 
\begin{lem}
\label{lem:multi-arc_slide}
Let $\arcs$ be a multi-basis on $\Pg$ and $\arcs'$ be obtained from $\arcs$ by a multi-arc slide. Then $\ordn(\Pg,\phi,\arcs')\leq\ordn(\Pg,\phi,\arcs)$.
\end{lem}
\begin{proof}
This follows readily from Lemma \ref{lem:incl-arcs} and Proposition \ref{prop:triangle_collapse}.
\end{proof}

\begin{cor}
\label{cor:hopf}
Let $\B:=(S,\phi)$ be an open book decomposition, and $\B':=(S',\phi')$ be a positive stabilization of $\B$. Then $\bbord(\B')\leq\bbord(\B)$.
\end{cor}
\begin{proof}
Let $\arcs$ be a multi-basis such that $\bbord(\B)=\ordn(S,\phi,\arcs)$. By a sequence of multi-arc slides, pass to a multi-basis $\arcs'$ on $\Pg$ that is disjoint from $c$. Then $\ordn(S,\phi,\arcs')=\ordn(S,\phi,\arcs)$ by Lemma~\ref{lem:multi-arc_slide} since $\bbord(\B)=\ordn(S,\phi,\arcs)$, and $\ordn(S',\phi',\arcs')=\ordn(S,\phi,\arcs')$ by Lemma \ref{lem:incl-surface} since $\arcs'$ is disjoint from $c$. As a result, 
\[\bbord(\B')\leq\ordn(\Pg',\phi',\arcs')=\ordn(\Pg,\phi,\arcs')=\ordn(S,\phi,\arcs)=\bbord(\B).\qedhere\]
\end{proof}

We move on to analyze the behavior of $\ord$ under Legendrian surgery.
\begin{prop} \label{prop:leg} Let $(\Pg,\phi)$ be an open book decomposition and $\arcs$ be any collection of pairwise disjoint properly embedded arcs on $\Pg$ that contains a basis. Suppose $c$ is a non-separating simple closed curve on $\Pg$ which meets each arc in $\phi(\arcs)$ at most once. Then $\ordn(\Pg, \tau_c \circ \phi, \arcs) \geq \ordn(\Pg, \phi, \arcs)$.
\end{prop}

\begin{proof}
To start with, use $(\Pg, \phi, \arcs)$ and the curve $c$ to form a multipointed triple Heegaard diagram $(\Sigma, \bsb, \bsg, \bsa, \base)$ where $\bsg=\{\gamma_1,\dots,\gamma_k\}$ with $\gamma_i=b'_i\times\{\frac{1}{2}\}\cup \tau_c \circ \phi(b'_i)\times\{0\}$ such that $b'_i$ is obtained from $b_i$ by slightly pushing along $\partial \Pg$ in the direction of the boundary orientation as in Figure \ref{fig:legtriple}.

Notice that $(\Sigma, \bsb, \bsa, \base)$ is the multipointed Heegaard diagram associated to $(\Pg, \phi, \arcs)$ and $(\Sigma, \bsg, \bsa, \base)$ is the multipointed Heegaard diagram associated to $(\Pg, \tau_c \circ \phi, \arcs)$. Meanwhile, the multipointed Heegaard diagram $(\Sigma, \bsb, \bsg)$ describes the manifold $\#_{\G} S^1 \times S^2$. Note also that the open book decomposition $(\Pg, \tau_c)$ together with the collection of arcs $\{b_1, \dots, b_k\}$ specifies the Heegaard diagram $(\Sigma, \bsg, \bsb)$ as in \cite{HKM}. The chain complex $\cfhat(\Sigma, \bsb, \bsg)$ has trivial differential and the generator $\vtheta$ indicated in Figure \ref{fig:legtriple} is the topmost generator. 

\begin{figure}[h]
\centering
\includegraphics[width=3in]{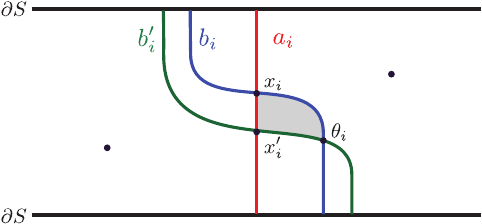}
\caption{Part of the restriction of the multipointed Heegaard triple diagram $(\Sigma,\bsb,\bsg,\bsa,z)$ to $S\times\{\frac{1}{2}\}\subset\Sigma$.}
\label{fig:legtriple}
\end{figure}

The placement of the basepoints guarantees, once again, that the multipointed triple Heegaard diagram $(\Sigma, \bsb, \bsg, \bsa, \base)$ is admissible. Now consider the chain map  
\begin{equation}
\label{eq:surgchainmap}
\hat{f}_{\bsb,\bsg,\bsa}:\cfhat(\Sigma,\bsb,\bsg)\otimes_{\F} \cfhat(\Sigma,\bsg,\bsa)\to \cfhat(\Sigma,\bsb,\bsa),\end{equation}
induced by the cobordism described by the triple Heegaard diagram $(\Sigma, \bsb, \bsg, \bsa)$. This chain map is non-trivial only in the canonical $\spinc$ structure $\mathfrak{t}_\omega$ corresponding to the Stein structure $\omega$ on the Legendrian surgery cobordism. This $\spinc$ structure restricts to the canonical $\spinc$ structures $\spc_\xi$ and $\spc_{\xi'}$ before and after surgery, respectively. Moreover, as is depicted in Figure \ref{fig:surgregion}, the differential on $\F\cdot\vtheta\otimes_{\F} \cfhat(\Sigma,\bsg,\bsa,\spc_{\xi'})$ is identically zero; hence, it is a subcomplex of $\cfhat(\Sigma,\bsb,\bsg)\otimes_{\F} \cfhat(\Sigma,\bsg,\bsa)$. Restricting \eqref{eq:surgchainmap} to this subcomplex, we obtain a chain map
\[\hat{f}_{\bsb,\bsg,\bsa;\mathfrak{t}_\omega}(\vtheta\otimes\cdot):\cfhat(\Sigma,\bsg,\bsa,\spc_{\xi'})\to \cfhat(\Sigma,\bsb,\bsa,\spc_\xi).\]
In fact, the $J_+$-filtered differential on $\F\cdot\vtheta\otimes_{\F} \cfhat(\Sigma,\bsg,\bsa,\spc_{\xi'})$ is identically zero since all homology classes in $\pitwo(\vtheta,\cdot)$ have the same $J_+$ value (see Figure \ref{fig:surgregion}). 
\begin{figure}[b]%
\includegraphics[width=.5\columnwidth]{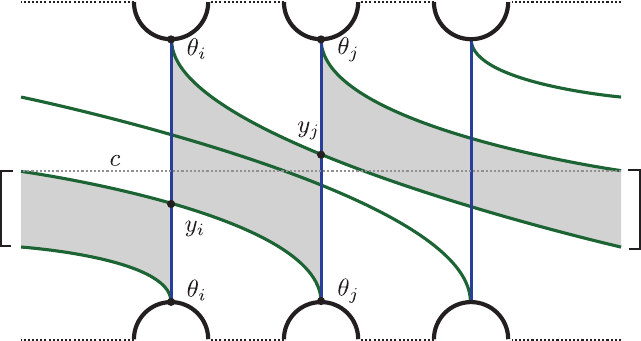}%
\caption{A local picture of $-\Pg\times\{0\}\subset\Sigma$ part of the Heegaard diagram $(\Sigma, \bsb, \bsg)$ near the surgery curve and all intersecting arcs. The shaded domains representing pseudo-holomorphic curves with negative punctures at $\vtheta$ have the same $J_+$ value. The brackets indicate that the ends of the shaded region connect to one another.}%
\label{fig:surgregion}%
\end{figure}
Therefore, having decomposed the above chain map as
\[\hat{f}_{\bsb,\bsg,\bsa;\mathfrak{t}_\omega}(\vtheta\otimes\cdot)=f^0+f^1+\cdots+f^\ell+\cdots,\]
where $f^\ell$ counts embedded Fredholm index-$0$ pseudo-holomorphic curves with $J_+=2\ell$, we have  
\begin{equation}
\label{eq:identity-2}
\sum_{i+j=\ell}(f^i\circ\partial^{'}_{j}-\partial^{}_{i}\circ f^j)=0.
\end{equation}
just as in Section \ref{sec:independence}. The identity \eqref{eq:identity-2} implies that there is a filtered chain map from $(\CCF(S,\tau_c\circ\phi,\arcs),\widehat{\partial}')$ to $(\CCF(S,\phi,\arcs),\widehat{\partial})$ and hence a morphism of spectral sequences from $E^\ast(S,\tau_c\circ\phi,\arcs;J'_{\mathit{HF}})$ to $E^\ast(S,\phi,\arcs;J^{}_{\mathit{HF}})$. In addition, $\hat{f}_{\bsb,\bsg,\bsa;\mathfrak{t}_\omega}(\vtheta\otimes\vx'_\xi)=\vx_\xi$ since the shaded triangle in Figure \ref{fig:legtriple} is the only holomorphic domain that contributes to this chain map due to the placement of the basepoints, and it is represented by a unique pseudo-holomorphic curve by the Riemann Mapping Theorem. Hence, $\ordn(\Pg, \tau_c \circ \phi, \arcs; \acs') \geq \ordn(\Pg, \phi, \arcs; \acs^{})$ as desired.
\end{proof}

\begin{cor} 
\label{cor:leg}
Let $\B:=(\Pg,\phi)$ be an open book decomposition and suppose $\B':=(\Pg,\phi')$ is obtained from $\B$ by Legendrian surgery, i.e. $\phi' = \tau_{c_n} \circ \cdots \circ \tau_{c_1} \circ \phi$. Then 
\begin{equation}
\label{eq:leg}
\bbord(\B) \leq \bbord(\B').
\end{equation}
As a consequence, if $\B:=(\Pg,\phi)$ is an open book decomposition where $\phi$ can be written as a product of positive Dehn twists, then $\bbord(\B)=\infty$.
\end{cor}

\begin{proof} 
We will apply Proposition \ref{prop:leg} one Dehn twist at a time, noting that for each Dehn twist curve $c_i$ we can find a multi-basis $\arcs$ on $\Pg$ so that $c_i$ intersects each arc in the image of $\arcs$ under the monodromy at most once. With the preceding understood, for each $i\in\{0,1,\dots,n\}$ denote by $\B_{i}$ the open book decomposition $(\Pg,\phi_{i})$ where $\phi_0=\phi$ and $\phi_{i}=\tau_{c_{i}} \circ \cdots \circ \tau_{c_1} \circ \phi$ for $i\in\{1,\dots,n\}$. For each $i\in\{1,\dots,n\}$, fix a multi-basis $\arcs_i$ on $\Pg$ such that $\bbord(\B_i)=o(\Pg,\phi_i,\arcs_i)$. Performing a sequence of multi-arc slides, pass to a multi-basis ${\arcs'}_{i}$ on $\Pg$ such that $c_i$ intersects each arc in $\phi_{i-1}({\arcs'}_{i})$ at most once. It follows from Lemma~\ref{lem:multi-arc_slide} and Proposition \ref{prop:leg} that 
\[\bbord(\B_{i-1})\leq o(\Pg,\phi_{i-1},{\arcs'}_{i})\leq o(\Pg,\phi_i,{\arcs'}_{i})=\bbord(\B_i).\]
Concatenating these inequalities for $i\in\{1,\dots,n\}$ while noting that $\B_0=\B$ and $\B_n=\B'$, we achieve the first claim of the corollary.

The last claim of the corollary follows immediately from \eqref{eq:leg} once we note that $\bbord(\Pg,id_\Pg)=\infty$. The latter is because the $J_+$-filtered differential in the corresponding Heegaard Floer chain complex is zero.
\end{proof}

With all the results needed in place, we are ready to prove Theorem \ref{cor:cobord}, and the second bullet point of Theorem \ref{thm:main}.

\begin{proof}[Proof of Theorem \ref{cor:cobord}]
Let $(M',\xi')$ be obtained from $(M,\xi)$ by Legendrian surgery. It suffices to prove this for Legendrian surgery on a single curve. Suppose without loss of generality that the Legendrian $c$ lies on a page of an open book decomposition $\B$ of $M$ supporting $\xi$. Positively stabilize $\B$ to get to an open book decomposition $\B_1$ which realizes $\ord(M,\xi)$; namely, $\bbord(\B_1) = \ord(M,\xi)$. Now consider the open book decomposition $\B_2$ of $M'$ supporting $\xi'$ obtained by Legendrian surgery on $c$. Positively stabilize $\B_2$ to get an open book decomposition $\B'$ with $\bbord(\B') = \ord(M',\xi')$. Now, mirroring these stabilizations on $\B_1$, we obtain an open book decomposition $\B_1'$ of $M$ supporting $\xi$ which, after Legendrian surgery, gives $\B'$. By Corollary \ref{cor:hopf} we have 
\[\bbord(\B_1') \leq \bbord(\B_1) = \ord(M,\xi),\] which implies that $\bbord(\B_1') = \ord(M,\xi)$; and by Corollary \ref{cor:leg} we have 
\[\bbord(\B_1') \leq \bbord(\B') = \ord(M',\xi').\] Hence 
\[\ord(M,\xi) = \bbord(\B'_1) \leq \bbord(\B') = \ord(M',\xi').\qedhere\]
\end{proof}

\begin{cor}
\label{cor:stein}
Let $(M,\xi)$ be Stein fillable. Then $\ord(M,\xi)=\infty$.
\end{cor}

\begin{proof}
A Stein fillable contact 3-manifold admits a supporting open book decomposition $(S,\phi)$ where $\phi$ is a product of positive Dehn twists. To be more explicit, a Stein fillable contact 3-manifold can be obtained via Legendrian surgery on some connected sum $\#_N S^1\times S^2$ equipped with its standard contact structure $\xi_{\it std}$  (see \cite{Gompf}). Therefore, by the second bullet point of Theorem~\ref{thm:main}, it suffices to prove that $\ord(\#_N S^1\times S^2,\xi_{\it std})=\infty$. To see this, let $\B$ be an open book decomposition of $\#_N S^1\times S^2$ supporting $\xi_{\it std}$ which realizes $\ord(\#_N S^1\times S^2,\xi_{\it std})$; in other words, $\bbord(\B)=\ord(\#_N S^1\times S^2,\xi_{\it std})$. As $(\#_N S^1\times S^2, \xi_{\it std})$ is supported by an open book with trivial monodromy, a common stabilization, $\B'$, of that and $\B$ will have a monodromy which can be written as a product of positive Dehn twists and will also realize the minimal $\ordn$. To see this, note that by the second claim in Corollary \ref{cor:leg}, we have $\bbord(\B')=\infty$. By Corollary \ref{cor:hopf}, we also have $\bbord(\B')\leq \bbord(\B)=\ord(\#_N S^1\times S^2,\xi_{\it std})$. Therefore, $\ord(\#_N S^1\times S^2,\xi_{\it std})=\infty$.
\end{proof}

Next, we prove the third bullet point of Theorem \ref{thm:main}: 

\begin{theorem}
\label{thm:multiarcs} 
Given an open book decomposition $(\Pg, \phi)$ of $M$ supporting $\xi$, and a basis $\arcs$ on $\Pg$, there exists a multi-basis $\arcs^m$ on $\Pg$ containing $\arcs$ such that
\[\ordn(\Pg, \phi, \arcs^m) = \ord(M,\xi).\]
\end{theorem}

\begin{proof}
Given an open book decomposition $\B = (\Pg, \phi)$ of $M$ supporting $\xi$, positively stabilize it to pass to an open book decomposition $\B' = (\Pg', \phi')$ with $\Pg'$ built from $\Pg$ by adding 1-handles, and $\phi' = \tau_{c_n} \circ \cdots \circ \tau_{c_1} \circ \phi,$ such that $\B'$ realizes $\ord(M,\xi)$; that is, $\bbord(\B') = \ord(M,\xi)$. Extending $\phi$ to $\Pg'$ as the identity on all the 1-handles, we form the open book decomposition $\tilde{\B} = (\Pg', \phi)$. Since $\phi'$ is obtained from $\phi$ by adding positive Dehn twists, $\bbord(\tilde{\B}) \leq \bbord(\B')$ by Corollary~\ref{cor:leg}. 

Now, fix a multi-basis $\arcs'$ on $\Pg'$ such that 
\[\ordn(\Pg', \phi', \arcs') = \bbord(\B') = \ord(M,\xi).\]
Let $a_1,\dots,a_n$ denote the co-cores of the 1-handles added to $\Pg$ so as to build $\Pg'$ and perform a sequence of multi-arc slides so as to pass to a multi-basis $\arcs''$ that contains the arcs $a_1,\dots, a_n$ and satisfies $\ordn(\Pg',\phi,\arcs'')=\bbord(\tilde{\B})$. We also have $\ordn(\Pg', \phi', \arcs')=\ordn(\Pg', \phi', \arcs'')$ by Lemma~\ref{lem:multi-arc_slide}. Let $\arcs^\circ=\arcs'' \cap \Pg$ and note that $\arcs^\circ$ is a multi-basis on $\Pg$. Furthermore, $\phi$ acts trivially on all arcs in $\arcs'' \smallsetminus \arcs^\circ$. Looking at the Heegaard diagram resulting from $(\Pg',\phi,\arcs'')$, the $\alpha$ and $\beta$ curves corresponding to arcs in $\Pg'\smallsetminus\Pg$ intersect each other exactly twice, forming two canceling bigons and thus contribute zero to $\dhat$. Furthermore, $\alpha_i$ and $\beta_i$ intersect no other $\bsa$-curves or $\bsb$-curves. Thus 
\[\CCF(\Pg', \phi, \arcs'') \equiv \CCF(\Pg', \phi, \arcs^\circ) \otimes_{\F} (\F_{(0)}\oplus\F_{(1)})^{\otimes n},\]
where $\F_{(0)}\oplus\F_{(1)}$ is a graded module over $\F$ with vanishing differential and $n$ is the number of arcs in $\arcs'' \smallsetminus \arcs^\circ$. In particular, 
\[\ordn(\Pg', \phi, \arcs'') = \ordn (\Pg', \phi, \arcs^\circ)\]
By Lemma \ref{lem:incl-surface}, we have $\ordn(\Pg, \phi, \arcs^\circ) = \ordn(\Pg', \phi, \arcs^\circ)$. Consequently,
\[\ordn(\Pg, \phi, \arcs^\circ) = \ordn(\Pg', \phi, \arcs^\circ) = \ordn(\Pg', \phi, \arcs'') \leq \ordn(\Pg', \phi', \arcs'') = \ordn(\Pg', \phi', \arcs') = \ord(M,\xi).\]
Since by definition $\ordn(\Pg, \phi, \arcs^\circ)\geq\ord(M,\xi)$, we have $\ordn(\Pg, \phi, \arcs^\circ)=\ord(M,\xi)$. Finally, given a multi-basis $\arcs$ on $\Pg$, perform a sequence of multi-arc slides to pass from $\arcs^\circ$ to a multi-basis $\arcs^m$ on $\Pg$ containing $\arcs$. Then by Lemma~\ref{lem:multi-arc_slide} $\ordn(\Pg, \phi, \arcs^m)=\ordn(\Pg, \phi, \arcs^\circ)=\ord(M,\xi)$, which finishes the proof.
\end{proof} 

\begin{rmk}
Note that given an open book decomposition $(\Pg,\phi)$ and a multi-basis $\arcs$ on $\Pg$, we can positively stabilize $(\Pg,\phi)$ to pass to a new open book decomposition where $\arcs$ becomes a basis. Then it follows from Corollary \ref{cor:hopf} and Theorem \ref{thm:multiarcs} that $\ord(M,\xi)=\ordn(\Pg,\phi,\arcs)$ for some open book decomposition $(\Pg,\phi)$ supporting the contact structure $\xi$ and a basis $\arcs$ on $\Pg$.
\end{rmk}

Another application of the Legendrian surgery statement in Theorem \ref{cor:cobord} is Theorem \ref{thm:consum}, that the spectral order of a contact connected sum is the minimum of the orders of the summands:

\begin{proof}[Proof of Theorem \ref{thm:consum}] Let $\B_1=(\Pg_1,\phi_1)$ and $\B_2=(\Pg_2,\phi_2)$ be open book decompositions which realize $\ord(M_1,\xi_1)$ and $\ord(M_2,\xi_2)$, respectively. Fix multi-bases $\arcs_1$ and $\arcs_2$ on $\Pg_1$ and $\Pg_2$, respectively, such that $\bbord(\B_i)=o(\Pg_i,\phi_i,\arcs_i)$ for $i=1,2$. Then both $\CCF(S_1,\phi_1,\arcs_1)$ and $\CCF(S_2,\phi_2,\arcs_2)$ can be seen as filtered subcomplexes of $\CCF(S_\#,\phi_\#,\arcs_\#)$ where $\B_1\#\B_2=(S_\#,\phi_\#)$ is the boundary connected sum open book decomposition with $\phi_\#=\phi_2\circ\phi_1$, where we extend each by the identity across the complementary subsurface, and $\arcs_\#=\arcs_1\sqcup\arcs_2$. Hence, by Lemmas \ref{lem:incl-arcs} and \ref{lem:incl-surface},
\[\ord(M_1\# M_2,\xi_1 \# \xi_2) \leq \bbord(\B_1\#\B_2) \leq \bbord(\B_i) = \ord(M_i,\xi_i),\]
for both $i=1$ and $i=2$, and $\ord(M_1\# M_2,\xi_1 \# \xi_2)\leq\min\{\ord(M_1,\xi_1),\ord(M_2,\xi_2)\}$.

For the reverse inequality, let $\B=(S,\phi)$ be a stabilization of $\B_1 \# \B_2$ realizing $\ord(M_1\# M_2,\xi_1 \# \xi_2)$. Ignore the extra positive Dehn twists on $\B$ which arise from its description as a positive stabilization of $\B_1 \# \B_2$. The resulting open book decomposition $\B'=(S,\phi')$ describes the 3-manifold $M_1 \# M_2 \#_k S^1 \times S^2$ for some $k$, the page $S$ contains $\Pg_1\#\Pg_2$ as a subsurface due to $\B$ being a positive stabilization of $\B_1 \# \B_2$, and the monodromy $\phi'$ extends $\phi_\#$ as the identity to the rest of $\Pg$. In particular, $\B$ is obtained from $\B'$ by Legendrian surgery along curves contained in a page of $\B$; hence, 
\[\ord(M_1\# M_2,\xi_1 \# \xi_2) = \bbord(\B) \geq \bbord(\B'),\] 
by the second bullet point of Theorem~\ref{thm:main}.

Fix a multi-basis $\arcs'$ on $S$ such that $\bbord(\B')=o(\Pg,\phi',\arcs')$. After a sequence of multi-arc slides, we can pass to a multi-basis $\widetilde{\arcs}$ on $\Pg$ which contains $\arcs_1\sqcup\arcs_2$. By Lemma~\ref{lem:multi-arc_slide} $\bbord(\B')= o(S,\phi',\widetilde{\arcs})$, and we have
\[\CCF(S,\phi',\widetilde{\arcs}) \cong \CCF(S_1,\phi_1,\arcs_1) \otimes_{\F} \CCF(S_2,\phi_2,\arcs_2) \otimes_{\F} (\F_{(0)} \oplus \F_{(1)})^{\otimes k},\] 
as filtered chain complexes where $\F_{(0)}\oplus\F_{(1)}$ is a graded module over $\F$ with vanishing differential. As a result, $\bbord(\B')= o(S,\phi',\widetilde{\arcs})=\min\{\ordn(\Pg_1,\phi_1,\arcs_1), \ordn(\Pg_2,\phi_2,\arcs_2)\}$. On the other hand, since $\ordn(\Pg_1,\phi_1,\arcs_1) = \bbord(\B_1) = \ord(M_1,\xi_1)$ and $\ordn(\Pg_2,\phi_2,\arcs_2) = \bbord(\B_2) = \ord(M_2,\xi_2)$, by the above inequality we have
\[\min\{\ord(M_1,\xi_1), \ord(M_2,\xi_2)\} \leq \ord(M_1\# M_2,\xi_1\# \xi_2).\]
\end{proof}

\begin{cor}
\label{cor:monoid}
For any surface $S$ with boundary, the set of monodromies yielding open book decompositions supporting contact 3-manifolds $(M,\xi)$ with $\ord(M,\xi)\geq k$ forms a monoid in the mapping class group $\mathrm{Mod}(S, \partial S)$. 
\end{cor}

\noindent We use $\ord^k(S)$ to denote this monoid. 

\begin{proof} By \cite{BakEtVHM}, for any two mapping classes $\phi_1$ and $\phi_2$, there is a Stein cobordism starting at the disconnected contact manifold $(M_{\phi_1}, \xi_{\phi_1})  \sqcup (M_{\phi_2}, \xi_{\phi_2})$ and ending at $(M_{\phi_2 \circ \phi_1}, \xi_{\phi_2 \circ \phi_1})$. By Theorems \ref{cor:cobord} and \ref{thm:consum}, this implies that \[\ord(M_{\phi_2 \circ \phi_1}, \xi_{\phi_2 \circ \phi_1}) \geq \ord\left((M_{\phi_1}, \xi_{\phi_1})  \sqcup (M_{\phi_2}, \xi_{\phi_2})\right) = \min \left\{ \ord(M_{\phi_1}, \xi_{\phi_1}), \ord(M_{\phi_2}, \xi_{\phi_2})\right\}.\qedhere\]\end{proof}

%%%%%%%%%%%%%%%%%%
\newpage
\section{An example}
\label{sec:example}
%!TEX root = Filtering.tex
In this section, we present a contact 3-manifold $(Y,\xi)$ with non-zero Ozsv\'ath--Szab\'o contact class but with zero spectral order. The contact structure $\xi$ is supported by the open book decomposition $(\Pg,\phi)$ where $\Pg$ is a compact oriented genus-1 surface with two boundary components and $\phi=\tau_a\tau_b{\tau_c}^{-1}$, the product of positive Dehn twists around the curves $a$, $b$ and a negative Dehn twist around the curve $c$ indicated in Figure \ref{fig:genus-1}. The contact structure $\xi$ has non-zero Ozsv\'ath--Szab\'o contact class by \cite[\textsection 4]{Conway} (cf. \cite[Corollary~4]{HP}). We will show that $\ord(Y,\xi)=0$. 

\begin{figure}[h]
\centering
\begin{subfigure}[b]{0.4\linewidth}
\centering
\includegraphics[width=2in]{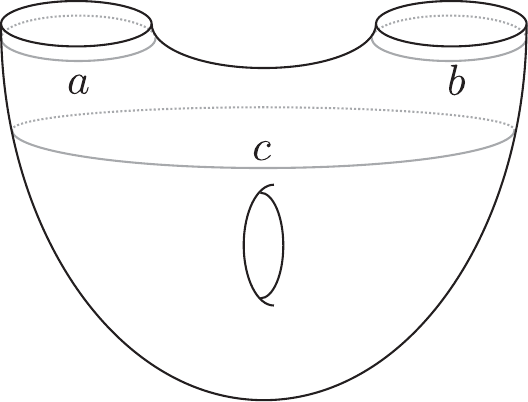}
\caption{}
\label{fig:genus-1}
\end{subfigure}
\begin{subfigure}[b]{0.4\linewidth}
\centering
\includegraphics[width=2in]{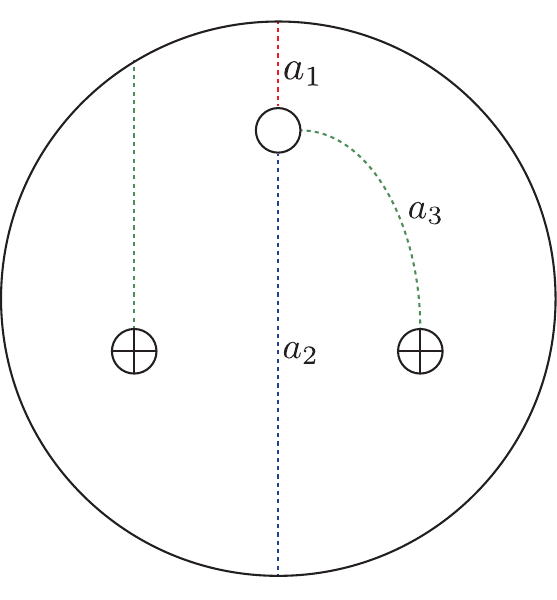}
\caption{}
\label{fig:arcbasis}
\end{subfigure}%
\caption{On the left is the open book decomposition $(\Pg,\phi)$ supporting the contact 3-manifold $(Y,\xi)$. On the right, the basis of arcs $\arcs=\{a_1,a_2,a_3\}$ on $\Pg$, where the two middle circles intersecting $a_3$ are identified.}
\end{figure}

To show that $\ord(Y,\xi)=0$, we need to find a multi-basis $\arcs$ on $\Pg$ such that $\ordn(\Pg,\phi,\arcs)=0$; more explicitly, we will find a generator $\vy$ of the resulting Heegaard Floer chain complex such that $\del_0\vy=\vx_\xi$. As we will show, it suffices to work with the basis of arcs $\{a_1,a_2,a_3\}$ depicted in Figure \ref{fig:arcbasis}. The effect of the monodromy on this basis of arcs is shown in Figure~\ref{fig:monpage}. In what follows, a region without a basepoint will be denoted by $R_i$ if it is numbered $i$ in Figure~\ref{fig:monpage}. 

\begin{figure}[h]
\centering
\includegraphics[width=4in]{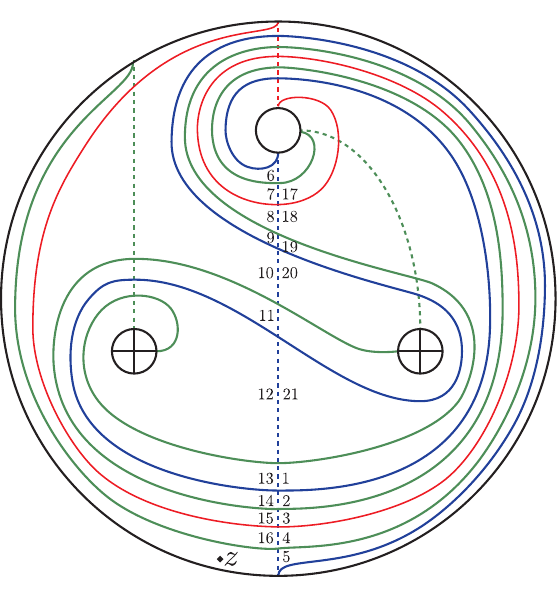}
\caption{The effect of the monodromy applied to the basis of arcs. The resulting regions without basepoint are numbered $1,\dots,21$.}
\label{fig:monpage}
\end{figure}

We claim that the generator $\vy$ determined by the tuple of intersection points $\y=(x_1,y_2,y_3)$ satisfies $\del_0\vy=\vx_\xi$ (see Figure \ref{fig:zerodomain}). In this regard, note that any $J_+=0$ holomorphic domain is an immersed $2n$-gon with all acute corners by \eqref{eq:jplus} and \eqref{eq:jplus-2}. 

\begin{figure}[h]
\centering
\includegraphics[width=4in]{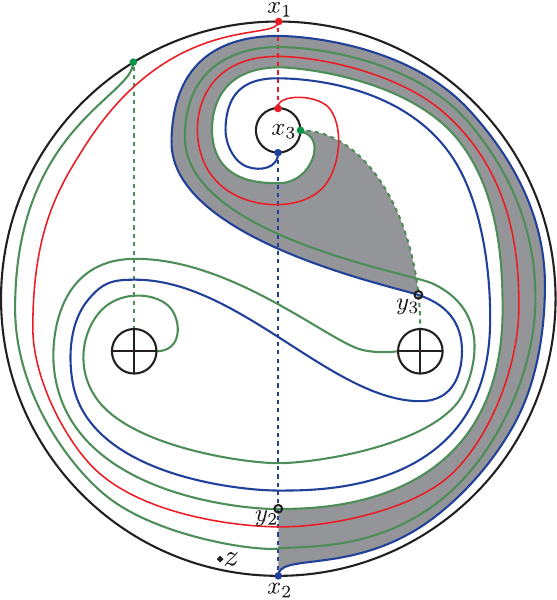}
\caption{The domain $\D_0$ (shaded).}
\label{fig:zerodomain}
\end{figure}

A positive Maslov index-1 $J_+=0$ domain $\D_0$ from $\y$ to $\x_\xi$ is shaded in Figure \ref{fig:zerodomain}. As a formal sum of regions without basepoints in the Heegaard diagram, it is given by
\[\D_0=R_3+R_4+R_5+R_7+R_8+R_9+R_{17}+R_{18}+R_{19}.\]
This domain is an embedded rectangle. Therefore, it has a unique holomorphic representative for a generic split almost complex structure by the Riemann Mapping Theorem. In fact, this is the only domain that represents a positive class in $\pitwo(\vy,\vx_\xi)$. This is because any other domain from $\y$ to $\x_\xi$ has to differ from $\D_0$ by a \emph{periodic domain} representing a \emph{periodic class} in $\pitwo(\vy,\vy)$. The latter is isomorphic to $H_2(Y;\mathbb{Z})$ which is a free Abelian group of rank~$2$. A basis for $\pitwo(\vy,\vy)$ is given by the following periodic domains:
\begin{eqnarray*}
P_1&=&R_1+R_4+R_5-R_6-R_7-R_{10}-R_{11}-R_{14}-R_{15}+R_{18}+R_{19}+R_{21},\\
P_2&=&R_2-R_5+R_6-R_9+R_{11}+R_{13}+2R_{14}+R_{15}+R_{16}-R_{17}-R_{18}-2R_{19}\\
&&-R_{20}-R_{21}.
\end{eqnarray*}
If $\D_0+aP_1+bP_2$ is a positive domain from $\y$ to $\x_\xi$, then in particular the following inequalities are satisfied:
\[0+a\geq 0,\quad 0+b\geq 0,\quad 0-a\geq 0,\quad 0-b\geq 0,\]
via the multiplicities of the regions $R_1$, $R_2$, $R_{10}$, and $R_{20}$ respectively. As a result, $a=0=b$. 

Next, we argue that there are no other positive Maslov index-1 $J_+=0$ domains from $\y$. To see this, let $\D$ be such a domain from $\y$ to some $\vv$ defining a generator of the Heegaard Floer chain complex, and move along the boundary of $\D$ in its boundary orientation. Note firstly that, due to the placement of the basepoint, $\D$ cannot have a corner at $x_1$. Therefore, $\D$ must be an immersed rectangle as none of the regions are bigons. As a result, $\vv = (x_1,v_2,v_3)$ for some $v_i \in \alpha_i\cap\beta_i$, $i=2,3$. Note further that the region $R_{12}$ adjacent to $y_3$ is an immersed 8-gon. Being a positive Maslov index-1 $J_+=0$ domain with four corners, $\D$ must have Euler measure $e(\D)=0$. As Euler measure is additive under unions and $\D$ is a positive domain, $\D$ cannot contain the region $R_{12}$, which has Euler measure $-1$. Hence, $\D$ contains only the region $R_{19}$ among the four regions adjacent to $y_3$. Now, $v_3\neq x_3$ since otherwise $\D=\D_0$. Moreover, as $\D$ does not contain the region $R_0$ with basepoint, it contains $R_{19}$ with multiplicity $1$ and does not contain the region $R_{18}$. This forces $\D$ to be contained in formal sum 
\[R_5+R_9+R_{19},\]
as $\D$ cannot contain the 6-gon regions $R_1$ and $R_{10}$, which have Euler measure $-\frac{1}{2}$. But then, $\D$ cannot have a corner at $y_2$, which is a contradiction.

Consequently, we have $\del_0\vy=\vx_\xi$ which implies that $\ord(Y,\xi)=o(\Pg,\phi,\arcs)=0$ as the spectral order is defined to be the minimum over all choices of open book decompositions $(\Pg,\phi)$ supporting $\xi$ and multi-bases $\arcs$ on $\Pg$. Consequently, by the second bullet point of Theorem \ref{thm:main}, $(Y,\xi)$ is not Stein fillable.

\begin{rmk}
In fact, $\dhat\vy=\del_0\vy+\del_1\vy=\vx_\xi+\vw$ where $\vw$ is determined by the tuple of intersection points $\w=(x_1,w_2,w_3)$ (see Figure \ref{fig:onedomain}). The domain $\D_1$ from $\y$ to $\w$ shown in Figure \ref{fig:onedomain} is an embedded genus-1 surface with one boundary component and $J_+(\D_1)=2$ given by the formal sum
\[\D_1=R_{11}+R_{12}+R_{13}+R_{14}.\]
\begin{figure}[h]
\centering
\includegraphics[width=4in]{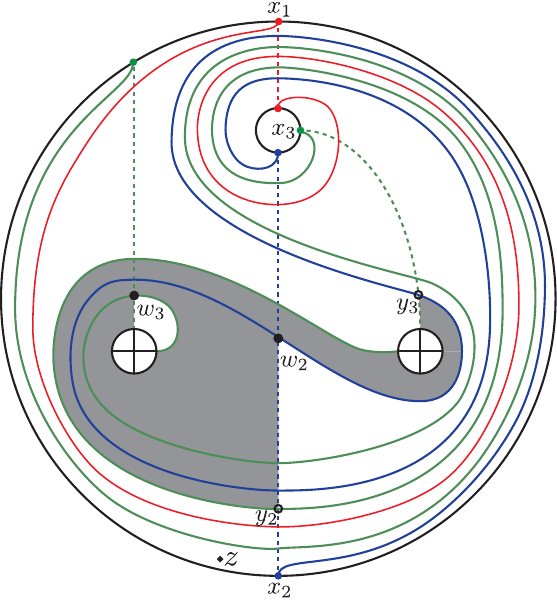}
\caption{The domain $\D_1$ (shaded). Keep in mind that the middle two circles are identified.}
\label{fig:onedomain}
\end{figure}

Arguing similarly to before, we see that if $\D_1+aP_1+bP_2$ is another positive domain from $\y$ to $\w$, then in particular the following inequalities are satisfied:
\[0+a\geq 0,\quad 0+b\geq 0,\quad 0-a\geq 0,\quad 0-b\geq 0,\]
via the multiplicities of the regions $R_1$, $R_2$, $R_7$, and $R_9$ respectively. As a result, $a=0=b$. Furthermore, a slightly more general version of the argument above proves that there are no other positive Maslov index-1 domains from $\y$. In particular, the domain $\D_1$ has a unique (up to a signed count) holomorphic representative since otherwise $\dhat\vy=\vx_\xi$, contradicting non-vanishing of the Ozsv\'ath--Szab\'o contact class.
\end{rmk}

%%%%%%%%%%%%%%%%%%
\bibliography{biblio}
\end{document}